\documentclass[10pt]{amsart}
\usepackage{amscd}

\usepackage{amssymb}

\newcommand{\bT}{{\mathbb T}}

\newcommand{\bZ}{{\mathbb Z}}
\newcommand{\bR}{{\mathbb R}}

\newcommand{\bA}{{\mathbb A}}
\newcommand{\bC}{{\mathbb C}}
\newcommand{\bF}{{\mathbb F}}

\newcommand{\bG}{{\mathbb G}}

\newcommand{\II}{{I_{\infty}^2}}

\newcommand{\la}{{\langle}}
\newcommand{\ra}{{\rangle}}

\newtheorem{thm}{Theorem}[section]
\newtheorem{lemma}[thm]{Lemma}
\newtheorem{cor}[thm]{Corollary}
\newtheorem{prop}[thm]{Proposition}
\newtheorem{as}[thm]{Assumption}
\newtheorem{conj}[thm]{Conjecture}

\newcommand{\Om}{{\Omega^*}}

\DeclareMathOperator{\Res}{Res}

\numberwithin{equation}{section}

\begin{document}

\title[Algebraic cobordism]{Algebraic cobordism 
and flag varieties}
 
\author{Nobuaki Yagita}

\address{ faculty of Education, 
Ibaraki University,
Mito, Ibaraki, Japan}
 
\email{ nobuaki.yagita.math@vc.ibaraki.ac.jp, }


\begin{abstract}
Let $X$ be an algebraic variety over $k$ such that
$\bar X=X\otimes _k\bar k$ is cellular.
We study torsion elements in the Chow ring $CH^*(X)$
which correspond to $v_iy$ in the algebraic cobordism
$\Omega^*(\bar X)$  where  $0\not=y\in CH^*(\bar X)/p$
and $v_i$ is the generator of $BP^*$ with $|v_i|=-2(p^i-1).$
In particular, we try to compute $CH^*(X)$
from $\Omega^*(\bar X)$  when $X$ are 
twisted complete flag varieties.

\end{abstract}

\maketitle
    
\section{Introduction}

Let $X$ be a smooth algebraic variety over a field $k$
of $ch(k)=0$ such that $\bar X=X\otimes _k\bar k$ is
cellular.  For a fixed prime $p$,
let $CH^*(X)=CH^*(X)_{(p)}$ be the Chow ring
generated by algebraic cycles modulo rational equivalence,
localized at $p$.  
Let us write 
\[\Omega^*(X)=MGL^{2*,*}(X)\otimes _{MU^*}BP^* \]
the $BP^*$-version of the algebraic cobordism 
 defined by Voevodsky ([Vo1], [Le-Mo1,2], [Ya2,6])
with   the coefficient ring
$\Omega^*=BP^*\cong \bZ_{(p)}[v_1,v_2,...]$,
where $|v_n|=-2(p^n-1)$.
The relation between
these theories are given ([Vo1],[Le-Mo1,2])
\[ CH^*(X)\cong \Omega^*(X)\otimes _{\Omega^*}\bZ_{(p)}.\]
In this paper, we study 
$CH^*(X)$ by the restriction map $res_{\Omega}:\Omega^*(X)\to \Omega^*(\bar X)$.

 Let $(y_i)=\{y_1,...,y_s\}$ be a $\bZ_{(p)}$-base of $CH^*(\bar X)$ 
with the degree $|y_i|\le |y_{i+1}|$ so that 
\[CH^*(\bar X)\cong \oplus _i\bZ_{(p)}\{y_i\},\quad
\Omega^*(\bar X)\cong \oplus _i\Omega^*\{y_i\}.\]
(Here $A\{ y_i\}$ is the $A$-free module generated by $y_i$.)
  Define the filtration
$F^i=\Omega^*\{y_j|i\le j\}\subset \Omega^*(\bar X)$ and the associated graded algebra
\[gr^*\Omega(\bar X)=\oplus_i F^i/F^{i+1}\cong 
\Omega^*\otimes CH^*(\bar X).\]
Let us write 
\[ Res_{\Omega}(X)=gr^*(Im(res_{\Omega}))\]
\[=\oplus_i(Im(res_{\Omega})\cap F^i)/(Im(res_{\Omega})\cap F^{i+1})
\subset gr^*\Omega^*(\bar X).\]

 An ideal $J\subset BP^*
$ is called invariant if $r(J)\subset J$ for all
(Landweber-Novikov)
cohomology operations $r$ in $\Omega^*(X)$ theory. 
By using the Cartan formula for the Landweber-Novikov
operation, it is almost immediate ;
\begin{lemma}  Let $(y_i)=\{y_1,...,y_s\}$ be a $\bZ_{(p)}$-base of $CH^*(\bar X)$.
Then there are 
  invariant ideals $J(y_i)$ in $\Omega^*$
such that
\[ Res_{\Omega}(X)\cong \oplus_iJ(y_i)\{y_i\}\subset 
\Omega^*\otimes CH^*(\bar X).\]
\end{lemma}

In the above lemma, let us write  $J(y_i)=(a_{i_1},..., a_{i_s})\subset BP^*$.
Then $a_{i_j}y_i$ is an $\Omega^*$-module generator in $Res_{\Omega}(X)$. Hence there is
an element $c_{i_j}$ in 
$\Omega^*(X)\otimes _{\Omega^*}\bZ_{(p)}
\cong CH^*(X)$ such that $res_{\Omega}(c_{i_j})=a_{i_j}y_i $ in $gr\Omega^*(\bar X)$.

Let us write
\[
Q(X)=P_{\infty}(X)=\oplus_iC(y_i) \quad 
with  
 \quad C(y_i)=\bZ/p\{c_{i_1},...,c_{i_s}\}\]
where the degree is given 
$ |a_{i_j}y_i|=|c_{i_j}|.$

\begin{cor}  For a base $(y_i)$ of $CH^*(\bar X)
$ and the invariant ideals
 $J(y_i)=(a_{i_1},...,a_{i_s})$, 
there is a  map
$ p_{\infty}: Q(X)\to CH^*(X)/p.$
\end{cor}
Note that the above map $p_{\infty}$ need not be injective.
There is a case that  for some $a_{i_j}$ with $p_{\infty}(c_{i_j})=0$ but
$a_{i_j}y_i\not \in F^i$,  that is,
\[(*)\quad  a_{i_j}y_i\in I_{\infty}(Im(res_{\Omega})\cap
\Omega^* \{\oplus_k y_k| k<i\}.\]
However $p_{\infty}$ seems surjective 
\begin{conj}  We have 
$Im(p_{\infty})\cong CH^*(X)/p$.
\end{conj}

However, it seems difficult to compute $Q(X)$ directly.
Hence we consider more easy versions.
The prime invariant ideals are written as 
$\Omega^*$   or  $ I_{n}=(p,v_1,....,v_{n-1})$ for  $0\le n\le \infty.$
We consider modulo $(I_{\infty}^{n+1})$ version
\[ Res_{\Omega, n}=Res_{\Omega}(X)/(Res_{\Omega}(X)\cap I_{\infty}^{n+1}gr\Omega^*(\bar X))\]
\[ \qquad \subset ( gr^*\Omega^*(\bar
X))/I_{\infty}^{n+1}\cong \Omega^*/I_{\infty}^{n+1}\otimes 
CH^*(\bar X),\]
and let us write by $P_n(X)$ the corresponding $Q(X)$
so that there is the commutative diagram
\[ \begin{CD}
P_1(X) @>>> P_2(X) @>>> ... @>>> P_{\infty}(X)=Q(X)@. \\
@VVV  @VVV  @VVV @VVV @.\\
Im(p_1)@>{inj.}>> Im(p_2)@>{inj.}>> ... @>{inj.}>> Im(p^{\infty})
@>{inj.}>> CH^*(X)/p.
\end{CD} \]

We give an example.
\begin{lemma} Let $y\not =0 \in CH^*(\bar X)/p$
with $J(y)=I_{n}$.  Then we have
\[ C(y)\cong M_n(y)=
\bZ/p\{c_0(y),c_1(y), ... , c_{n-1}(y)\},
\]
with degree $ |c_{k}(y)|=-2p^k+2+|y|.$
\end{lemma}

If $v_iy\in Res_{\Omega}(X)$, then we  see
that $v_jy\in Res_{\Omega}(X)$ for all $j\le i$.
So define $deg_v(y)=n$ if $v_{n-1}y \in Res_{\Omega}(X)$ but
$v_ny\not \in Res_{\Omega}(X)$. (Let $deg_v(y)=0$ if $y\in Res_{\Omega}(X)$ and $deg_v(y)=-1$ if $py\not \in Res_{\Omega}(X)$, that means $M_0=\bZ/p$ and $M_{-1}=0$.)  Then we can write
\begin{cor}
Let $CH^*(\bar X)\cong \bZ_{(p)}\{y_1,...,y_s\}$ and 
$|y_i|\not =|y_j|$ for $i\not =j$.  Then 
\[P_1(X)\cong   \oplus_i M _{d_i}(y_i) 
\quad where \ d_i=deg_v(y_i).\]
\end{cor}

In general, to compute $deg_v(y)$
is not  a so easy problem.  However there are cases $deg_v(y)$ are known or computable. 

Given   a pure symbol $a$ in the mod $p$  Milnor $K$-theory
$ K_{n+1}^M(k)/p$, by Rost,  we can
construct the norm variety $V_a$ of $dim(V_a)=p^n-1$  (such that $\bar V_a$ is cellular).
Rost and Voevodsky  showed that
there is $y\in CH^{b_n}(\bar V_a)$ for
$b_n=(p^n-1)/(p-1)$  such that $y^{p-1}$ is the fundamental class of $\bar V_a$. 
Moreover,  there is an irreducible (Rost) motive  $R_a$ ( write it by $R_n$
simply)
 in the motive $M(V_a)$ of  $V_a$
such that
$\  CH^*(\bar R_a)\cong \bZ_{(p)}[y]/(y^p)$.
We see   $deg_v(y^i)=n$ for all $1\le i\le p-1$, and moreover
we know
([Ro1], [Vo1,4], [Me-Su], [Vi-Ya], \cite{YaB})
\begin{thm}
 We have $ CH^*(R_n)/p\cong P_1(R_n)$ and
\[ P_1(R_n)\cong
M_0\{1\}\oplus \oplus_{i=1}^{p-1}M_n(y^i)
\cong \bZ/p\{1\}\oplus  M_n(y)\{1,y,...,y^{p-2}\}.\]
\end{thm}

We consider the K-theory version of $P_i(X)$.
Let $k^*=\bZ_{(p)}[v_1]$.
Take $J(y)'=J(y)\otimes_{BP^*}k^*$ instead of $J(y)$
(so that  $I_n'=I_2=(p,v_1)$ for $n\ge 2$).   Let us write
the corresponding $P_i(X)$ by $P(1)_i(X)$. For example, when $n\ge 2$,
\[ P(1)_1(R_n)\cong \bZ/p\{1\}\oplus  M_2(y)\{1,y,...,y^{p-2}\}.\]
Recall that  $gr_{geo}^*(X)$ is the  graded ring
associated with the geometric filtration of the algebraic $K$-theory
$K^0(X)$ ([To3], \cite{YaG}).
\begin{lemma}
If $ K^0(X)\cong  K^0(\bar X),$ then
we have a surjective map
\[ p(1)_{\infty}:  P(1)_{\infty}(X)\to gr_{geo}^*(X)/p.\]
\end{lemma}

Another examples which we can compute 
$deg_v(y)$ are
the cases that $X=\bG/B_k$ twisted complete 
versal flag
varieties.  The degree $deg_v(y)$ are computed by using the Milnor operations $Q_i$.

Let $G$ be a simply connected compact Lie group
and $T$ its maximal torus.
By Borel ([Bo], [Mi-Tod]), its $mod(p)$ cohomology is written as 
\[ (*)\quad gr H^*(G;\bZ/p)\cong P(y)/p\otimes \Lambda(x_1,...,x_{\ell}),
\quad \ell=rank(G)\]
\[ \quad  with  \ \  P(y)=\bZ_{(p)}[y_1,...,y_k]/(y_1^{p^{r_1}},...,
y_k^{p^{r_k}})
\]
where the degree $|y_i|$ of $y_i$ is even and 
$|x_j|$ is odd. 
(Here  $\Lambda(a,...,b)$ is the $mod(p)$ exterior algebra generated by $a,...,b$.)

 Let  $BT$ be the classifying space of $T$.
We consider the fibering ([Tod], \cite{Mi-Ni})
$ G\stackrel{\pi}{\to}G/T\stackrel{i}{\to}BT$
and the induced spectral sequence 
\[ (**)\quad E_2^{*,*'}=H^*(BT;H^{*'}(G;\bZ/p)) \Longrightarrow H^*(G/T;\bZ/p).\] 
The cohomology  of $BT$ is given 
 by
$H^*(BT)\cong S(t)=\bZ[t_1,...,t_{\ell}]$ with $|t_i|=2$.
It is well
known that $y_i$ are permanent cycles (i.e., exist in $E_{\infty}^{0,*'}$) and 
that there is a regular sequence ([Tod], \cite{Mi-Ni})
$(\bar b_1,...,\bar b_{\ell})$ in $H^*(BT)/(p)$ such that $d_{|x_i|+1}(x_i)=\bar b_i$.  (These $\bar b_i$ are called the transgressive elements.)   Thus we get
\[ E_{\infty}^{*,*'}\cong grH^*(G/T;\bZ/p)\cong P(y)/p\otimes 
S(t)/(\bar b_1,...,\bar b_{\ell}).\]
 Moreover we know that $G/T$ is a manifold 
such that $H^*(G/T)$ is torsion free, and 
\[grH^*(G/T)\cong P(y)\otimes S(t)/(b)\quad for \  
S(t)/(b)=S(t)/(b_1,...,b_{\ell}) \]
where $ b_i=\bar b_i\ mod(p).$
Since $H^*(G/T)$ is torsion free,  the 
Atiyah-Hirzebruch 
spectral sequence (AHss) collapses.  Hence  we also know 
\[gr BP^*(G/T)\cong BP^*\otimes grH^*(G/T). \]

The $v_n$-module structure of $\Omega^*(X)$
relates the $Q_n$-module structure of
$H^*(X;\bZ/p)$ for the Milnor operation
\[ Q_n : H^{*}(X;\bZ/p)\to H^{*+2p^n-1}(X;\bZ/p).\]
\begin{lemma} 
In the spectral sequence $(**)$, 
let $d_{|x_i|+1}(x_i)=b_i\not =0$.
 Then we have the relation
in $ BP^*(G/T)/\II$ such that 
\[ b_i=py(0)+v_1y(2)+...+v_ny(n)+...\]
where $y(k)\in H^*(G/T;\bZ/p)$ with $\pi^*y(k)=Q_kx_j$
for $\pi: G\to G/T$.
\end{lemma}

Let $G_k$ be the split reductive algebraic  group corresponding to $G$, and $T_k$ be the split maximal
torus corresponding to $T$.  Let $B_k$ be the Borel subgroup
with $T_k\subset B_k$.   Note that $G_k/B_k$ is cellular, and
$CH^*(G_k/T_k)\cong CH^*(G_k/B_k)$.
Hence we have   
$ CH^*(G_k/B_k)\cong H^{2*}(G/T)$
and $CH^*(BB_k)\cong H^{2*}(BT).$

Let $\bG$ be a nontrivial $G_k$-torsor.
We will study some motive in  the twisted flag variety $\bF=\bG/B_k$.
Moreover let $\bG$ be $versal$ 
(For the definition and properties of versal, see
$\S 10$ below  or  [Ga-Me-Se], [To2], [Me-Ne-Za], [Ka1].)

The versal case  of the main result in Petrov-Semenov-Zainoulline \cite{Pe-Se-Za} is
given  as follows  ;
Let $\bG$ be a versal $G_k$-torsor, and   $\bF=\bG/B_k$. 
Then there is a $p$-localized  submotive $R(\bG)$ of the motive $M(\bF)$ of $\bF$ such that 
\[CH^*(\bF)\cong CH^*(R(\bG))\otimes S(t)/(b),
\quad CH^*(\bar R(\bG))\cong P(y).\]
Moreover, each element in 
$CH^*(R(\bG))$ is written as a sum
of products $b_i$, since the map $CH^*(BB_k)\to
CH^*(\bF)$ is surjective when $\bG$ is versal
([Ka1], [Me-Ne-Za]).

From the preceding lemma, we have
\begin{lemma}  Suppose $deg_v(y)>0$ for
$y\in P(y)/p$ and $d_{|x|+1}(x_j)\not =0$.  If  in $H^*(G;\bZ/p)$,
\[Q_n(x_j)=y  \quad  but \quad
Q_k(x_j)=0\ \ for \ 0\le k< n\]
then $deg_{v}(y)\ge n+1$.
\end{lemma}

In the last sections in this paper, by using the above lemma, 
we try to compute
\[  P(1)_{\infty}^*(R(\bG))\to gr_{geo}^*(R(\bG))/p \]
for $(G,p)=(Spin(13),E_7; p=2)$ and $(E_8;p=3)$.

{\bf Example.}  Let  $G=Spin(11)$ (see $\S 12$). Then 
\[grH^*(G;\bZ/2)\cong P(y)\otimes \Lambda(x_3,x_5,x_7,x_9,z_{15})\]
 with $P(y)=\bZ_{(2)}[y_6,y_{10}]/(y_6^2,y_{10}^2)$.
Here the suffix means its degree.  
We denote $c_i=b_{i-1}=d(x_{2i-1})$ for $2\le i\le 5$, and $e_8=b_5
=d(z_{15})$.
The $Q_i$-actions are well known e.g.,
$Q_0(x_3)=0$, $Q_1(x_3)=Q_0(x_5)=y_6$.
From the preceding lemma and Lemma 1.8, the restriction map $res_{\Omega}$ is given
\[ \ c_2\mapsto v_1y_6,\  c_3\mapsto 2y_6, \ 
 \ c_4\mapsto v_1y_{10}, \  c_5\mapsto 2y_{10}, 
\] \[
\ c_2c_4\mapsto v_1^2y_6y_{10},\  e_{8}\mapsto 2y_6y_{10}.\]
Hence we see $J(y_6)=J(y_{10})=(2,v_1)$. Moreover we can show $J(y_6y_{10})=(2,v_1^2)$.  Thus we can see
\[ Res_{\Omega}(R(\bG))\ \cong\  BP^*\{1\}\oplus (2,v_1)\{y_6,y_{10}\}
\oplus (2,v_1^2)\{y_6y_{10}\}\]
in $gr\Omega^*(\bar R(\bG))\cong \Omega^*\otimes P(y)$. Then $C(y_6)\cong \bZ/2\{c_3,c_2\}$ for example. 
Thus  we can see  
\[gr_{geo}^*(R(\bG))/2\cong
P(1)_2^*(R(\bG))\cong  
 C(1)\oplus C(y_{6})\oplus C(y_{10})\oplus 
C(y_6y_{10})\]
\[ \cong 
 \bZ/2\{1, c_2,c_3,c_4,c_5,c_2c_4, e_8\}.\]
Moreover, 
Karpenko ([Ka2]) proves $gr_{geo}^*(R(\bG))\cong 
CH^*(R(\bG)).$

This paper is organized as follows.
In $\S 2$ we recall the motivic  $BP$-theory
using the Atiyah-Hirzebruch spectral sequence.
In $\S 3$, we define $deg_v(y)$ and study the case $J(y)=I_{n}$.
We generalize it to general invariant ideals in $\S 4$.    In $\S 5$, we study the graded ring $gr_{geo}^*(X)$ from $K$-theory.
In $\S6,7$, the norm variety and the Dickson invariants are studied, considering relations to the operation $Q_i$.
In $\S 8$, we study $deg_v(y)$ when $X$ is an anisotropic  quadric.
In $\S 9$,   we recall
 the cohomology and $BP^*$-theory 
of  topological flag manifolds $X=G/T$.
In  $\S 10,$  we recall the results
by  Petrov, Semenov and
Zainoulline, for the motves  $R(\bG)$ of flag varieties $\bG/B_k$.  In $\S 11, \S 12$, we study $Res_{\Omega}(R(\bG))$ when $G=SO(odd)$ and $G=Spin(odd)$.
We study $gr_{geo}^*(R(\bG))/2$ for groups 
for $Spin(7),Spin(9), Spin(11)$ in $\S 13$.  The group
$Spin(13)$ is studied in $\S 14$.
The exceptional groups $E_7,E_8 ; p=2$ are studied 
in $\S 15, \S16$. The case $E_8, p=3$ is studied in $\S17$.

\section{algebraic $BP$ theories}

Recall that $MU^*(-)$ is the complex cobordism theory defined 
in the usual (topological) spaces and 
\[MU^*=MU^*(pt.)\cong \bZ[x_1,x_2,...]\quad |x_i|=-2i.\]
Here each $x_i$ is represented by a sum of hypersurfaces 
of $dim(x_i)=2i$ in some product of complex projective spaces ([Ha],[Ra]).
Let $MGL^{*,*'}(-)$ be the motivic cobordism theory defined by
Voevodsky [Vo1]. 
Let us write by $AMU$ the spectrum $MGL_{(p)}$ 
in the stable $\bA^1$-category representing this motivic cobordism theory (localized at $p$), i.e., 
\[ MGL^{*,*'}(-)_{(p)}=AMU^{*,*'}(-).\]
Here note that $AMU^{2*,*}(pt.)\cong MU_{(p)}^{2*}$.
It  is not  isomorphic to $AMU^{*,*'}(pt)$  in general,
 while $AMU^{*.*'}(X)$ is an $MU_{(p)}^*$-algebra.

Given a regular sequence $S_n=(s_1,...,s_n)$ with $s_i\in MU^*_{(p)}$,
we can inductively construct the $AMU$-module spectrum by
the cofibering of spectra ([Ya2,4], [Ra])
\[(2.1)\quad \bT^{-1/2|s_i|}\wedge AMU(S_{i-1}) \stackrel{\times 
s_i}{\longrightarrow}
AMU(S_{i-1})\to AMU(S_i)\]
where $\bT=\bA/(\bA-\{0\})$ is the Tate object.
For the realization map $t_{\bC}$ induced from $k\subset \bC$, it is also immediate that $t_{\bC}(AMU(S_n))\cong MU(S_n)$ with
$MU(S_n)^*=MU^*/(S_n).$

Recall that the Brown-Peterson cohomology theory
$BP^*(-)$ with the coefficient  $BP^*\cong \bZ_{(p)}[v_1,v_2...]$ by
 identifying $v_i=x_{p^i-1}$. (So $|v_i|=-2(p^i-1)$.)
We can construct spectra (in the stable $\bA^1$-homotopy category)
\[ABP=AMU(x_i|i\not =p^j-1)\]
such  that $t_{\bC}(ABP)\cong BP$.
 For $S=(v_{i_1},...,v_{i_n})$, let us write
\[ABP(S)=AMU(S\cup\{x_i|i\not =p^j-1\})\]
so that $t_{\bC}(ABP(S))=BP(S)$ with $BP(S)^*=BP^*/(S)$.

In particular, let
$AH\bZ=ABP(v_1,v_2,...)$
so that $AH\bZ^{2*,*}(pt.)\cong \bZ_{(p)}$. 
 In the $\bA^1$-stable homotopy category,  
  Hopkins-Morel  showed that
\[ AH\bZ\cong H_{\bZ}\quad , i.e., \quad AH\bZ^{*,*'}(X)\cong H^{*,*'}(X,\bZ_{(p)})\]  
the (usual) motivic cohomology.  Using this result, we can construct the motivic Atiyah-Hirzebruch spectral sequence. 
  \begin{thm} ([Ya2,6])   Let $Ah=ABP(S)$ for 
  $S=(v_{i_1},v_{i_2},...)$, and  
recall  
\[ h^{2*''}=BP^{2*''}/(S)\cong Ah^{2*'',*''}(pt).\]
Then there is AHss (the Atiyah-Hirzebruch
spectral sequence)
\[E(Ah)_2^{(*,*',2*'')}=H^{*,*'}(X;h^{2*''})\Longrightarrow Ah^{*+2*'',*'+*''}(X)\]
with the differential \quad 
$d_{2r+1}:E_{2r+1}^{(*,*',2*'')} \to 
E_{2r+1}^{(*+2r+1,*'-r,2*''-2r)}$.
\end{thm}
Note that the cohomology $H^{m,n}(X,h^{2n'})$ here is the usual
motivic cohomology {with (constant) coefficients in} the abelian group $h^{2n'}$.   Here we recall some important properties
of the motivic cohomology. When $X$ is smooth, we know
\[ CH^*(X)\cong H^{2*,*}(X;\bZ_{(p)}),\quad
 H^{*,*'}(X;\bZ_{(p)})\cong 0\ for \  2*'<*.\] 

Hence  if $X$ is smooth, then
$E_r^{m,n,2n'}\cong 0$ for $m>2n.$
The convergence in AHss means that
there is the filtration by the third degree (by $v_i's$)
\[Ah^{*,*'}(X)=F_0^{*,*'} \supset F_1^{*,*'}\supset F_2^{*,*'}\supset
...\]
such that $F_i^{*,*'}/F_{i+1}^{*,*'}
\cong E_{\infty}^{*+2i,*'+i,-2i}.$

 Let $S\subset R=(v_{j_1},...)$.  Then
the  induced map $ABP(S)\to ABP(R)$ of spectra induces the 
$BP^*_{(p)}$-module map of AHss
:  $E(ABP(S))_r^{*,*',*''} \to E(ABP(R))_r^{*,*',*''}.$
In general, $ABP(S)^{*,*'}(X)\not \cong ABP^{*,*'}(X)/(S).$ 
However, from the above maps and dimensional reason, we see
for  a smooth $X$,
\[ (2.3)\quad ABP(S)^{2*,*}(X)\cong ABP^{2*,*}(X)/(S),\]
\[(2.4)\quad ABP(S)^{2*,*}(X)\otimes _{BP^*}\bZ_{(p)}\cong H^{2*,*}(X)\cong CH^*(X).\]

In this paper,  
a $connective $  $orientive $  theory $h^{2*}(X)$
means $ABP(S)^{2*,*}(X)$ as above.
We mainly consider 
the connective  oriented theory $ABP^{2*,*}(X)$.
We write it simply
\[ \Omega^{2*} (X)=ABP^{2*,*}(X)
\cong  MGL^{2*,*}(X)\otimes_{MU^*}BP^*.\]
Hence from (2.4), \ $\Omega^*(X)\otimes _{\Omega^*}\bZ_{(p)}\cong
CH^*(X)$ for a smooth $X$.

Recall the filtration $F_i^{*,*'}$ and the associated graded ring
\[(2.5)\quad gr\Omega^{2*}(X)\cong \oplus_{i}E_{\infty}^{2(*+i),*+i,-2i} \cong
\oplus_{i}E_{2}^{2(*+i),*+i,-2i}/Im(d) \]
\[\cong  \oplus_i(\Omega^{-2i}\otimes CH^{*+i}(X))/Im(d)\]
\[where \quad Im(d)=\cup  Im(d_{r}) \subset \Omega^{<0}\otimes CH^*(X)\cong E_2^{2*,*, <0}\]
 for the differential
in  the AHss converging to $ABP^{*,*'}(X)$. 
Note that $d_rE_r^{2*,*,0}=0$ and each element in 
$E_2^{2*,*,0}$ is permanent.

When $X$ is cellular, then $CH^*(X)=H^{2*,*}(X)$
is torsion free.
Since it is known the image of the differential of
the 
AHss is torsion.   Hence all
$d_r=0$ and $Im(d)=0$.  Thus the AHss collapses
\[gr\Omega^*(X)\cong \Omega^*\otimes 
CH^*(X)\quad for\ X:cellular.\]

Let $y_i$ generate  (as an $\Omega^*$-module) 
$\Omega^*( X)$  (hence $y_i\not =0\in CH^*(X)/p$).
Then from (2.5), we can write
\[ gr\Omega^*(X)\cong \oplus_i F^i/F^{i+1},  \quad for\ F^i=\sum_{|y_j|\ge i}\Omega^*(y_j).\]
Here,  we used the notation that 
$\Omega^*(y)$ means the $\Omega^*$-submodule generated by $y$ in $\Omega^*(X)$, 
(while $\Omega^*\{y\}$ means the $\Omega^*$-free module
generated by $y$.)

\section{Image $Res$ of the restriction}

Let us write $\bar X=X\otimes \bar k$ for the algebraically closure $\bar k$ of $k$.
Let  $res_{\Omega}: \Omega^*(X) \to \Omega^*(\bar X)$ be the restriction map.
Recall that $gr \Omega ^*(\bar X)\cong \Omega^*\otimes CH^*(\bar X)$ since $\bar X$ is cellular.

Let $\{y_1,...,y_s\}$ be a $\bZ_{(p)}$-base of $CH^*(\bar X)$ 
with the degree $|y_i|\le |y_{i+1}|$ so that 
\[CH^*(\bar X)\cong \oplus _i\bZ_{(p)}\{y_i\},\quad
\Omega^*(\bar X)\cong \oplus _i\Omega^*\{y_i\}.\]
(Here $A\{ y_i\}$ is an $A$-free module generated by $y_i$.)
  Define the filtration
$F^i=\Omega^*\{y_j|j\ge i\}\subset \Omega^*(\bar X)$ and the associated grade algebra
\[gr^*\Omega(\bar X)=\oplus_i F^i/F^{i+1}\cong 
\Omega^*\otimes CH^*(\bar X).\]
(This filtration $F^i$ is finer than that in the previous
section.)

Les us write the image of the restriction map
\[ Res_{\Omega}(X)=gr^*(Im(res_{\Omega}))\]
\[ =\oplus_{i}\ ( Im(res_{\Omega})\cap F^i)/
(Im(res_{\Omega})\cap F^{i+1}) \ \
\subset \ \  gr\Omega^*(\bar X).\]

Recall that $I_n=(p=v_0,v_1,...,v_{n-1})$ and 
$I_{\infty}=(v_0,v_1,...)$ are 
the prime invariant ideals (under Landweber-Novikov operations) of $BP^*=\Omega^*$.
\begin{lemma}
Let $y\in \Omega^*(\bar X)$ be an $\Omega^*$-module generator. 
If $v_ny\in Res_{\Omega}(X)$,  then $v_jy\in Res_{\Omega}(X)$ for all $0\le j\le n$,
 namely,
\[ I_{n+1}\{y\} \subset Res_{\Omega}(X)\subset  gr\Omega^*(\bar X).\]
\end{lemma}
\begin{proof}
Recall the Quillen (Landweber-Novikov) operation $r_{\alpha}$ in
the $\Omega^*(-)$ theory ([Ha], [Ra], [Ya2,4,7]) 
so that
\[ r_{p^i\Delta_{n-i}}(v_n)=v_i\quad mod(I_{i}^2).\]
Let $\alpha=p^i\Delta_{n-i}$. 
 Then by the  Cartan formula,  we have
\[ r_{\alpha}(v_nx)=r_{\alpha}(v_n)\cdot x +\sum_{\alpha'+\alpha''=\alpha,\ |\alpha''|>0}r_{\alpha'}(v_n)r_{\alpha''}(x).\]
For $x\not =0\in F^i/F^{i+1}$, the operation
 acts as $r_{\alpha''}(x)\in F^{i+1}$.
Hence in $gr\Omega^*( X)$ we get 
$r_{\alpha}(v_nx)=(v_i+a_i)x$ with $a_i\in I_{i}^2$.
When $i=0,1$, it is immediate we can take $a_i=0$.
By induction on $i$, we see
 if  $v_nx\in Res_{\Omega}(X)\subset gr\Om(X)$, then  so is  $v_ix$ for
$0\le i\le n$.  
\end{proof}

{\bf Definition.}
For a smooth $X$ and an $\Om$-module generator
$y\in \Om(\bar X)$ with $y\not \in Res_{\Omega}(X)$, we say that $y$ has $v$-degree $n$
\[deg_v(y)=n \quad  if \ v_{n-1}y\in \Res_{\Omega}(X)\ but\ 
v_{n}y\not \in Res_{\Omega}(X). \] 

When $y\in Res_{\Omega}(X)$, let $deg_v(y)=0$
and when $py\not \in Res_{CH}$ let $deg_v(y)=-1$.
\begin{lemma}
If $M$ is an irreducible $\Omega^*$-module, then
$M$ is contained in one irreducible motive in $\Omega^*(X)$.
\end{lemma}
\begin{proof}
Each map from $\Omega^*(X)$ to $\Omega^*(Y)$ of motives  is represented by a composition map
$f_*g^*$ of the induced map $g^*$ (of smooth varieties) and the Gysin map $f_*$.
The both maps are $\Omega^*$ maps.
Hence a decomposition of an $\Omega^*$-motive gives that of an $\Omega^*$-module. \end{proof}

\begin{lemma}    If $deg_vy_i=n>0$, then  there are $c_j\in CH^*(X)$ for $\ 0\le j\le n-1$ with $|c_j|=|y_i|-2(p^j-1)$ such that there is a map
\[ p_i: M_n(y)'=\bZ_{(p)}\{c_0\}\oplus \bZ/p\{c_1,....,c_{n-1}\} \to CH^*(X).\].
\end{lemma}
\begin{proof} We recall that
\[CH^*(X)\cong \Omega^*(X)\otimes_{\Omega^*} \bZ_{(p)}\cong
\Omega^*(X)/(\Omega^{<0}\Omega^*(X))\]
where $\Omega^{<0}=BP^{<0}=(v_1,...,v_n,...)$ is the ideal of $\Omega^*$
generated by negative degree elements.
The assumption $deg_v(y)=n$ implies $I_{n}\{y\}\subset 
\Omega^*\{y\}\cap Res_{\Omega}(X)$.  
Moreover since $v_jy\not \in Res_{\Omega}(X)$ for all $j\ge n$,
we see
\[ I_n\{y\}=\Omega^*\{y\}\cap Res_{\Omega}(X)
\mod ((BP^{<0})^2\Omega^*\{y\}).\]

Here we note
\[ I_n/((BP^{<0})\cdot  I_n+(BP^{<0})^2)\cong (p,...,v_{n-1})/((BP^{<0})\cdot
(p,...,v_{n-1}))\]
\[ \cong \bZ_{(p)}\{p\}\oplus \bZ/p\{v_1,...,v_{n-1}\}
\cong M_n(y)'.\]
Here $pv_i=0$ for $i\not =0$, since  $p\cdot v_i=v_i\cdot p \in BP^{<0}\cdot I_n$.

Note that 
\[CH^{\ge |y|}(X)\supset  (\Omega^*\{y\}\cap Res(X))/(I_{\infty}\Omega^*(X)\cap \Omega^*\{y\}),\]
\[M_n(y)'\cong (\Omega^*\{y\}\cap Res(X))/(I_{\infty}\Omega^*\{y\}).\]
Thus we have the map $M_n(y)' \to CH^*(X).$
\end{proof}
\begin{lemma}
Let $M$ be an irreducible motive in the motive 
$M(X)$ of smooth $X$.  Let  $|y_{i-1}|<|y_i| $ for $y_i\in CH^*(\bar M)$ (e.g., $ CH^{|y_i|}(\bar M)
\cong \bZ_{(p)}$). If $deg_v(y)=n>0$,
then $M_n(y)'$ is contained
in $CH^*(M)$ , i.e., the map $p_1$ is injective.
\end{lemma}
\begin{proof}
Let $J(y_i)=I_{n+1}$ and $v_{n}y_i\in Res(M)$.
Assume that $v_ny_i\in I_{\infty}\Omega^*(M)$
but not in $I_{\infty}\{y_{j}|j\ge i\}$.
Let us write 
\[ v_{n}y_i=pf_0+...+v_{n}f_{n}\quad in \ \Omega^*(\bar M),\]
where $f_k\in Im(res_{\Omega})\subset \Omega^*(\bar M)$ and
$f_k\not \in \Omega^*\{y_j|j\ge i\}$.  Hence 
\[ f_n=\sum _{k<i}v(k)y_k+y_i+\sum_{i< j} v(j)y_j\quad for \ v(k),v(j)\in \Omega^*.\]
Of course $f_n$ is homogeneous, and
$v(k)=0$ since $|v(k)|<0$.
Hence $y_k$ for $k<i$ dose not appear.
\end{proof}

Recall $M_n(y)=M_n(y)'/p$ is defined  for $n>0$.
Moreover, let $M_0(y)'=\bZ_{(p)}\{y\}$ and $M_{-1}(y)'=0$ (when
$py\not \in Res_{\Omega}(X)$).
\begin{cor}
Let $CH^*(\bar X)\cong \oplus_i\bZ_{(p)}\{y_i\}$
where $|y_i|\not =|y_j|$ for $i\not =j$.  Then $CH^*(X)/p$ has the submodule $P_1(X)$
such that
\[ P_1(X)=  \oplus_{i}M_{d_i}(y_i)\subset CH^*(X)/p \quad where \ d_i=deg_v(y_i).\]
\end{cor}
The injection of the above corollary is an isomorphic
when $X$ is the Rost motive defined by the norm variety.  But it is far from an isomorphism for
general cases.

Let us write
\[ P_1^N(X)=\oplus _{|y_i|\le N}M_{d_i}(y_i).\]
\begin{lemma} 
For $n\ge 0$, let $f:X\to Y$ be a map of smooth varieties such that
\[ f^*:CH^*(\bar Y)/p \to CH^*(\bar X)/p \ \text{
is injective for }\ *\le N,\]
and that $deg_v(f^*y)\ge 1$ 
for all  $0\not =y\in CH^*(\bar Y)/p$ such that 
$deg_v(y)\ge 1$ and $|y|\le N$.
Then $deg_v(y)\le deg_v(f^*(y))$ and
\[ f^* :  P_1^N(Y) \to P_1^{*}(X)\ \ 
\text{is injective}. \]
\end{lemma}
\begin{proof}
If $v_jy_i\in Res_{\Om}(Y)$, then
$ f^*(v_jy_i)=v_jf^*(y_i)\in Res_{\Om}(X).$
Hence $deg_v(y_i)\le deg_v(f^*(y_i)).$
For $ d\le d'$, it is immediate $M_{d}\subset M_{d'}$ by the definition.
\end{proof}
\begin{cor}
Let $CH^{*\le N}(\bar X)\cong \bZ_{(p)}\{x_1,...,x_s\}$
with $|x_i|\not =|x_j|$ for $i\not =j$.
Suppose the assumption of the preceding lemma.
Then the composition 
\[ P_1^{ N}(Y)\to P_1(X)\to CH^*(X)/p\]
injective for $*\le N$.
\end{cor}

\section{Another Invariant ideals}

In this section, we generalize arguments 
for the module $M_n(y)'/p=M_n(y)$ in the preceding section as stated in the introduction.

Recall that
$Res_{\Omega}(X)=gr^* (Im(res_{\Omega}))
\subset gr^*(\Omega^*(\bar X)).$
 An ideal $J\subset BP^*
$ is called invariant if $r(J)\subset J$ for all
(Landweber-Novikov)
cohomology operations $r$ in $\Omega^*(X)$ theory. 
By using the Cartan formula for the Landweber-Novikov
operation as the proof of Lemma 3.1, we see 
\begin{lemma}  Let $(y_i)$ be a $\bZ_{(p)}$-base of $CH^*(\bar X)$, i.e., $CH^*(\bar X)\cong \oplus _i\bZ_{(p)}\{y_i\}.$
Then there are 
  invariant ideals $J(y_i)$ in $\Omega^*$
such that
\[ Res_{\Omega}(X)\cong \oplus_iJ(y_i)\{y_i\}\subset 
\Omega^*\otimes CH^*(\bar X).\]
\end{lemma}
\begin{proof}
Let us write  $J(y_i)=(a_{i_1},..., a_{i_s})$ for $a_{i_j}\in \Omega^*$.
 Then using the Cartan formula as the proof of Lemma 3.1, we have 
\[r_{\alpha}(a_{ij}y_i)=r_{\alpha}(a_{ij})y_i+\sum_{|\alpha''|>0}r_{\alpha'}(a_{ij})r_{\alpha''}(y_i)\quad in \ F^{|y|}\subset
\Omega^*(\bar X).\]
Here for $y_i\in F^j$, we have 
 $r_{\alpha}(a_{ij})y_i$ in $F^{j}/F^{j+1}\subset
gr\Omega^*( \bar X).$
Hence if  
$a_{i_j}(y)\in Res_{\Omega}(X)$, then so is $r_{\alpha}(a_{ij})(y)$.  

Hence we see $J(y_i)$ is an invariant ideal
beecause  the Quillen operations $r_{\alpha}$ generate
the Landwber-Novikov operations, and all 
stable $BP^*$-module operations 
in $ABP^{*,*'}(X)$ theory [Ya2,4].
\end{proof}
 
In the above lemma, let us write  $J(y_i)=(a_{i_1},..., a_{i_s})$.
Then $a_{i_j}y_i$ is an $\Omega^*$-module generator in $Res_{\Omega}(X)$. Hence there is
an element $c_{i_j}$ in 
$\Omega^*(X)\otimes _{\Omega^*}\bZ_{(p)}
\cong CH^*(X)$ such that $res_{\Omega}(c_{i_j})=a_{i_j}y_i $ in $gr\Omega^*(\bar X)$.

Let us write
\[
Q(X)=P_{\infty}(X)=\oplus_iC(y_i) \quad 
with  
 \quad C(y_i)=\bZ/p\{c_{i_1},...,c_{i_s}\}\]
where the degree is given 
$ |a_{i_j}y_i|=|c_{i_j}|.$
That is 
\[ Q(X)=P_{\infty}(X)=Hom_{\Omega^*}(Res_{\Omega}(X),
\bZ/p), \quad \bZ/p=\Omega^*\otimes _{\Omega^*} \bZ/p\]
so that $C(J(y_i))=\bZ/p\{c_{i_1},...,c_{i_s}\}$ by
$c_{i_j}: a_{i_j}y_i \mapsto 1$.

\begin{cor}  For a base $(y_i)$ of $CH^*(\bar X)
$ and the invariant ideals
 $J(y_i)=(a_{i_1},...,a_{i_s})$, 
we have the   map
$ p_{\infty}: P_{\infty}(X)\to CH^*(X)/p.$
\end{cor}
Note that the above map $p_{\infty}$ need not be
injective, in general, because there is a possibility
such that for some $a_{i_j}$,
\[(4.1)\quad  a_{i_j}\in I_{\infty}(Im(res_{\Omega})\cap
\Omega \{\oplus_k y_k| k<i\}.\]
However $p_{\infty}$ seems surjective

\begin{conj}  W have 
$Im(q_{\infty})\cong CH^*(X)/p$.
\end{conj}

{\bf Example.}  When $J(y)=I_n=(p,...v_{n-1})$, we see 
\[C(y)=\bZ/p\{c_0,...,c_{n-1}\}\cong M_n(y)\quad with \ |c_i|=|v_iy|.\]

 Of course there are many other invariant ideals. For examples 
the ideal  $J=(p^2,pv_1,v_1^2)$ is invariant
but $J'=(p^2,v_1^2)$ is not for $p\not =2$.
In fact,
\[ r_{\Delta_1}(v_1^2)=2pv_1, \quad r_{2\Delta_1}(v_1^2)=p^2\]
where we used 
$r_{\Delta_1}(v_1)=p$ and   $r_{\alpha}(v_1)=0\ mod(I_{\infty}^2)$ for $\alpha\not =\Delta_1.$
(See [Ha] for detail examples of invariant ideals.)
If $J=J(y)$, then
\[ C(y)\cong \bZ/2\{c_0,c_1,c_2\}\quad with\
|c_0|=|y|, |c_1|=|v_1y|,|c_2|=|v_1^2y|.\]

However, it seems difficult to compute $Q(X)$ directly.
Hence we consider more easy version.
Recall that 
$ I_{n}=(p,v_1,....,v_{n-1})$.
We consider modulo $(I_{\infty}^{n+1})$ version
\[ Res_{\Omega, n}=Res_{\Omega}(X)/(Res_{\Omega}(X)\cap I_{\infty}^{n+1}gr\Omega^*(\bar X))\]
\[ \qquad \subset ( gr^*\Omega^*(\bar
X))/I_{\infty}^{n+1}\cong \Omega^*/I_{\infty}^{n+1}\otimes 
CH^*(\bar X),\]
and let us write by $P_n(X)$ the corresponding $Q(X)$ so that 
\[ P_{n}(X)=Hom_{\Omega^*}(Res_{\Omega,n}(X),
\bZ/p).\]

Hence there is the commutative diagram
\[ \begin{CD}
P_1(X) @>>> P_2(X) @>>> ... @>>> Q=P_{\infty}(X)@. \\
@VVV  @VVV  @VVV @VVV @.\\
Im(p_1)@>{inj.}>> Im(p_2)@>{inj.}>> ... @>{inj.}>> Im(p^{\infty})
@>{inj.}>> CH^*(X)/p.
\end{CD} \]

If  $J(y)= (p^2,pv_1,v_1^2)$, then $C(y)$ is contained
in $P_2(X)$. In particular, if $J(y)=(p^n)$, then
$C(y)=\bZ/p\{c_0\}$ is contained in $P_m(X)$
for $m\ge n$.

The following lemma is well known 
\begin{lemma} (Landweber [La])
Let $J$ be an invariant ideal.  There is $s\ge 0$
and $k\ge 2$  such that
\[J=(p,v_1,...,v_{s-1},v_s^k,...).\]
\end{lemma}
\begin{cor}
If $J$ is the invariant ideal in the above lemma.
Let us write $J_i=J/I_{\infty}^{i+1}$.Then
\[ J_1=...=J_{k-1}=I_s,\quad J_k=(p,...v_{s-1},v_s^k,...).\]
\end{cor}

\begin{lemma}
Let $M$ be an irreducible motive in the motive 
$M(X)$ of smooth $X$.  
Let  $|y_{i-1}|<|y_i|+|v_1|$, and let
 $J(y_i)=(p^2,pv_1,v_1^2)$ or $J(y_i)=(p,v_1^2)$. 
Then $p_2|_{C(y_i)}: C(y_i)\to CH^*(M)/2$
is injective
(that is, $C(y_i)$ is contained in $CH^*(M)/2$). 
\end{lemma}
\begin{proof}
We will prove the case $J(y_i)=(p,v_1^2)$, and the other cases are proved similarly.

Assume that $v_1^2y_i\in I_{\infty}\Omega^*(M)$
but not in $I_{\infty}\{y_{j}|i\le j\}$.
Let us write 
\[ v_{1}^2y_i=pf_0+v_{1}f_{1}\quad in \ \Omega^*(\bar M),\]
where $f_k\in Im(res_{\Omega})\subset \Omega^*(\bar M)$ and
$f_k\not \in \Omega^*\{y_j|i\le j\}$.  Hence 
\[ f_1=\sum _{k<i}v(k)y_k+v_1y_i+\sum_{i< j} v(j)y_j\quad for \ v(k),v(j)\in \Omega^*.\]
Of course $f_1$ is homogeneous, and
$v(k)=0$ since $|v(k)|<0$.
Hence $y_k$ for $k<i$ does not appear.
\end{proof}

\section{relation to $K$-theories}

Recall the (topological) Morava $K$-theory
\[K(n)^*(X)=BP(p,...,\hat v_n,...)^*(X)[v_n^{-1}]\]
is the generalized cohomology theory with the coefficient ring $K(n)^*=\bZ/p[v_n,v_n^{-1}]$. 
Note that when $n=1$, this Morava $K$-theory
is $essentially$ isomorphic to the usual $mod(p)$
$K$-theory [Ya5]. 
 Hence when we want to know $mod(p)$ $K$-theory $K^*(X;\bZ/p)$, we only need to know this Morava $K$-theory $K(1)^*(X)$.

For an algebraic space $X$, we consider
the integral algebraic Morava $K$-theory
$A\tilde K(n)^{*,*'}(X)$ and its connective 
theory $A\tilde k(n)^{*,*'}(X)$ with the coefficeint
rings
$ A\tilde k(n)^{2*,*}\cong\bZ_{(p)}[v_n],$
and $A\tilde K(n)^{2*,*}\cong \bZ_{(p)}[v_n.v_n^{-1}]$
so that \[A\tilde k(n)^{*,*'}(X)[v_n^{-1}]\cong
A\tilde K(n)^{*,*'}(X).\]

We consider the AHss given in $\S 2$
\[ (5.1)\quad E^{*,*',*''}_2(X)
\cong H^{*,*'}(X)\otimes \tilde K(n)^{*''}    
\Longrightarrow A\tilde K(n)^{*,*'}(X).\]
The associated graded ring is defined as the infinite term ([Ya5,6]) 
\[ gr(n)^*(X)=E_{\infty}^{2*,*,0}(X).\]

\begin{lemma} We have 
$CH^*(X)/I(n)\cong gr(n)^*(X),$
where $I(n)$ is the ideal generated by $v_n$-torsions
in
$E_{\infty}^{2*,*,0}$ for $A\tilde k(n)^{*,*'}(X)$.
\end{lemma}
\begin{proof}
Recall $CH^*(X)\cong H^{2*,*}(X)$ and 
$d_r(H^{2*,*}(X))=0$ for smooth $X$.
By the $A\tilde k(n)^{*,*'}(X)$ version of (5.1), we have
\[ gr A\tilde k(n)^{2*,*}(X)\cong
E_{\infty}^{2*,*,*'} \cong
 (\tilde k(n)^{2*,*}\otimes
CH^*(X))/Im(d).\]
Since $\tilde k(n)^{2*,*}\cong \bZ_{(p)}[v_n]$, 
an element in $Im(d)$ is written as $d_r(x)=v_n^sx'$ for $s\ge 1$ and  $x'\in CH^*(X)$.

In the AHss for $A\tilde k(n)^{*,*'}(X)$,
 the element 
$x'\in E_{\infty}^{2*,*,0}$ is a (higher) $v_n$-torsion element, indeed, $v_n^sx'=0$.
The AHss (5.1)  is the localization by $v_n$ of 
the AHss for $A\tilde k(n)^{*,*'}(X)$. So 
$x'=0\in E_{\infty}^{2*,*,0}$ in the AHss converging $A\tilde K^{*,*'}(X)$ (5.1).  
\end{proof}

Recall the definition of $Q(X)=P_{\infty}(X),P_i(X)$
in the introduction (or the preceding section). We consider their $\tilde k(n)$ version.
Let 
$\tilde k(n)^*=\tilde k(n)^{2*,*}\cong \bZ_{(p)}[v_n].$
Let us write
\[J'(y_i)=J(y_i)\otimes_{\Omega^*} \tilde k(n)^*\cong J(y_i)\otimes _{BP^*}\bZ_{(p)}[v_n] \]
(while it is not invariant ideal in general).
For example, $I_m'=(p,v_n)$ for $m>n$. 

When  $J'(y_i)=(a_{i_1}',...,a_{i_r}')\subset \tilde k(n)^*$, we take
$c_{i_1}',...,c_{i_r}'$ in $CH^*(X)/p$.  Then define
\[ P(n)_{\infty}(X) =\oplus_{i}\bZ/p\{c_{i_1}',...,c_{i_r}'\}.\]
Similarly we can define $P(n)_j(X)$ by considering
$J'(y_i)/I_{\infty}^{j+1}$ so that
there is the commutative diagram
\[ \begin{CD}
P(n)_1(X) @>>> P(n)_2(X) @>>> ... @>>> P(n)_{\infty}(X)@. \\
@VVV  @VVV  @VVV @VVV @.\\
Im(p(n)_1)@>{inj.}>> Im(p(n)_2)@>{inj.}>> ... @>{inj.}>> Im(p(n)_{\infty})
@>{inj.}>> gr(n)^*(X)/p.
\end{CD} \]

For ease of notations (and we do not topological arguments), let us write (throughout this section)
\[   \tilde K(n) ^*(X)=A\tilde K(n)^{2*,*}(X),\quad   h^*(X)=A\tilde k(n)^{2*,*}(X)
\quad  h^*= \bZ_{(p)}[v_n].\]

From (2.4), we see $h^*(X)\otimes_{h^*}\bZ_{(p)}
\cong CH^*(X).$
From Lemma 5.1, we see
\[gr(n)^*(X)/p   
\cong CH^*(X)/(p,I(n))\]
\[\cong (h^*(X)\otimes _{h^*}\bZ/p)/I(n)\cong h^*(X)/(I(n),I_{\infty}).\]
Hence, for  $a\in h^m(X)$, 
the following conditions are equivalent
\[a =0 \in gr(n)^m(X)/p \ \ 
\Longleftrightarrow  \ \ 
a\in (I_{\infty},I(n))h^*(X).\]

\begin{lemma}
Let $a\not =0\in \tilde K(n)^0(X)$.  Then there is
$s\ge 0$ and $a'\in h^*(X)$  such that 
\[a=v_1^sa' \quad  and \quad a'\not \in Im(I_{\infty}h^*(X)).\]
Hence $0\not = a'\in  gr(n)^{2s}(X).$
\end{lemma}
\begin{proof}
Suppose that  $a\not =0\in \tilde K(n)^0(X)$.
Since $\tilde K(n)^*(X)\cong h^*(X)[v_n^{-1}]$, there is
$s$ such that 
\[ a=v_n^{s}a'\quad for \ some \ a'\in h^{2s}(X).\]
Let us take the largest such $s$.
Then $a'$ is an $h^*/p=\bZ/p[v_n]$ module generator of $h^*(X)/p$
and $|a|=0$ but $|a'|\ge 0$.
So $s$ must be nonnegative.
\end{proof}

\begin{lemma}
Suppose $\tilde K(n)^0(X)\cong \tilde K(n)^0(\bar X).$  Then the map is surjective
\[p(n)_{\infty} : P(n)_{\infty}(X) \twoheadrightarrow gr(n)^*(X)/p.\] 
\end{lemma}  
\begin{proof}  Let
 $0\not =x\in gr(n)^*(X)\cong h^*(X)/(I_{\infty},I(n)).$ 
We consider the restriction $res_h: h^*(X)\to h^*(\bar X)$. 
Suppose that 
$res_h(x)\in I_{\infty}res(h^*(X)) $
(note $I(n)=0  \in h^*(\bar X)$, since it is $v_n$-free).
Then take $x'\in h^*(X)$ such that 
\[ res_h(x-v_n^sx') \in F^{k+1} \subset h^*(\bar X)
\quad where \ x'\in F^k.\]
Continuing this argument we see 
$ res(x-a)=0$  with  $a\in I_{\infty}h^*(X).$

Since $K^0(X)\cong K^0(\bar X)$, the restriction 
$res_{K}$ is isomorphic.  Hence  we see
$x=a$ in $h^*(X)/(I(n))$.
This means $x\in I_{\infty}h^*(X)$ and $x=0\in gr(n)^*(X)$, and this is a contradiction.
Thus we can take $x$ so that $res(x)$ is a generator 
of $Res_{h}(X)$.  This implies $x
\in P(n)_{\infty}(X)$ and from Lemma 5.2, we have this lemma.
\end{proof}

\section{Motivic cohomology of the norm variety}

Let $X$ be a smooth (quasi projective) variety.
Recall that  $H^{*,*'}(X;\bZ/p)$ is the $mod(p)$ motivic cohomology defined by Voevodsky and Suslin.
Rost and Voevodsky solved the Beilinson-Lichtenbaum
(Bloch-Kato) conjecture ([Vo2.4], [Su-Jo], [Ro2])
\[  H^{*,*'}(X;\bZ/p)\cong H_{et}^*(X;\mu_p^{\otimes *'})
\quad for \ *\le *'.\]
In this paper, we always assume that $k$ has
 a primitive $p$-th $root$ $of$
 $unity$ (so $\mu_p\cong \bZ/p$).
 Then  we have  the isomorphism
$H_{et}^m(X;\mu _p^{\otimes n})\cong H_{et}^m(X;\bZ/p).$ 
Let $\tau$ be a generator of  $H^{0,1}(Spec(k);\bZ/p)\cong \bZ/p$,
so that 
$ colim_{i}\tau^iH^{*,*'}(X;\bZ/p)\cong  H_{et}^*(X;\bZ/p).$

Let $\chi_X$ be the $\hat Cech$ complex, which is defined
as $(\chi_X)^n=X^{n+1}$ (see details [Vo1,2,4]).  Let
$\tilde \chi_X$ be defined as a cofiber
\[   \tilde \chi_X\to \chi_X \to Spec(k)\]
in the stable $\bA^1$-homotopy category.
Voevodsky defined the motivic cohomology $H^{*,*'}(\chi;\bZ/p)$
for all object $\chi$ in the stable $\bA^1$-homotopy category.
Morover, he showed that there exists the Milnor operation
in the motivic cohomology ([Vo1,3])
\[ Q_i: H^{*,*'}(\chi;\bZ/p)\to H^{*+2p^i-1,*'+p^i-1}(\chi;\bZ/p)\]
which is compatible with the usual Milnor operation
on $H^*(t_{\bC}(\chi);\bZ/p)$ for the realization map $t_{\bC}$.
(Here the topological operation is defined $Q_0=\beta$
Bockstein operation and $Q_{i+1}=Q_i{P^{p^i}}-P^{p^{i}}Q_i$.)
This operation $Q_i$ can be extended on $H^{*,*'}(M;\bZ/p)$
for a motive $M$ in $M(X)$ (Lemma 7.1 in \cite{YaB}).

For $0\not =x\in H^{*,*'}(X;\bZ/p)$ (or cohomology operation),
let us write $*=|x|, w(x)=2*'-*, d(x)=*-*'$ so that $0\le d(x)\le dim(X)$
and $w(x)\ge 0$ when $X$ is smooth.  We also note
\[ |\tau|=0,\  w(\tau)=2,\quad |Q_i|=2p^i-1,\  w(Q_i)=-1.\]

Let $f_X\in CH^{dim(X)}(\bar X)$ be the fundamental class
of $\bar X$.  Suppose $deg_v(f_X)=n$, i.e., $v_{n-1}f_X\in Res_{\Omega}(X)$.  It is known that
the ideal 
\[ I(X)=\pi_*(X)\subset BP^*(pt.)=BP^*\quad for\ \pi: X\to pt.\]
is an invariant ideal (e.g. Lemma 5.3 in \cite{YaB}).  Since
$\pi_*(f_X)=1$, we get $\pi_*(v_{n-1}f_X)=v_{n-1}$, and this means $I(X)\supset I_n$.  Thus we have
\begin{lemma} Let $f\in CH^*(X)$ be a fundamental class of $X$.  If  $deg_v(f)=n$, then $I_n\subset I(X)$.
\end{lemma}
Note that when $X=V $ is the norm variety 
so that $t_{\bC}(\bar V)=v_n$,
the ideal $I(V)=I_{n+1}$, infact, $\pi_*(1)=v_n$.
But $deg_v(f_V)=n$ e.g., $v_nf_V\not \in Res_{\Omega}(X)$.  (See Corollary 6.8 below.)

The following two lemmas are known.
\begin{lemma} (Lemma 6.4,6.5 in \cite{YaB})
If $I_{n}\subset I(X)$, then
$ABP^{*,*'}(\tilde \chi_X)$ is $I_n$-torsion,
and $H^{*,*'}(\tilde \chi_X;\bZ/p)$ is $\Lambda(Q_0,...,Q_{n-1})$-free.
\end{lemma}
\begin{lemma} (Corollary 3.4 in [Ya])
Let $x\in CH^*(X)\cong E_{\infty}^{2*,*,0}$ and $v_sx=0$ in $E_{\infty}^{2*,*,*'}$ for the
AHss converging to $ABP^{*,*'}(X)$.  
Then there is $b\in H^{*,*'}(X;\bZ/p)$ with
$Q_s(b)=x$  and  is a relation in $ABP^{2*,*}(X)=\Omega^*(X)$ 
\[ v_sx+v_{s+1}x_{s+1}+...+v_kx_k+...=0\quad mod(I_{\infty}^2)\]
with $x_k=Q_k(b)$ in $H^{2*,*}(X;\bZ/p)$ for all $k\ge s$.
\end{lemma}

We consider the following assumption
for an existence of the motive $M$ in $X$.
\begin{as}
 There is a motive $M$ in $M(X)$ such that
\[(1)\quad \text{there is}\ y\in CH^*(\bar M)\  with\ deg_v(y)\ge 1,\]
\[ (2)\quad H^{*,*'}(M;\bZ/p)\cong H^{*,*'}(\chi_X;\bZ/p)
\ \ for\ 0<*'<*\le |y|.\]
\quad (3) \quad  the cohomology $H^{*,*'}(\tilde \chi_X;\bZ/p)$ for $0<*'<*\le |y|$
is generated by only one element $a'$ as the
$ \quad K_M^*(k)\otimes \Lambda(Q_0,...,Q_{n-1})$-modules. 
\end{as}

\begin{thm}
Assume $deg_v(y)=n$ and $(1)-(3)$ as above.
Then  we have
\[   c_i(y)=Q_0...\hat Q_i... Q_{n-1}(a').\]
Moreover $CH^*(M)/p\cong C(y)\cong  M_n(y)$ for $0<*\le |y|$.
\end{thm}
\begin{proof}
Since $deg_v(y)=n$, there is $c_{n-1}(y)\in H^{2*,*}(M;\bZ/p)$
which correspond to $v_{n-1}y$.
Since $|Q_ic_{n-1}|\le |y|$ for $i\le n-2$, we see
$Q_i(c_{n-1})=0$ from (2) since $w(Q_ic_{n-1})=-1$.
Since $I(X)\supset I_n$, 
$H^{*,*'}(\tilde \chi_X;\bZ/p)$
is $\Lambda(Q_0,...,Q_{n-1})$-free from Lemma 4.5.
Therefore $c_{n-1}(y)$ is in the $Q_i$-image.
Hence there is $a''\in H^{*,*'}(\tilde \chi_X;\bZ/p)$
such that
\[ Q_0...Q_{n-2}(a'')=c_{n-1}(y).\]

Hence $a''=a'$ or $a''=Q_{n-1}a'$  from (3). 
(Note $w(c_{n-1})=0$ and $a''$ does not 
contains elements in $k_+^M(k)$ from (2)).

Suppose $a''=Q_{n-1}a'$.
Then $c_{n-1}(y)=Q_0...Q_{n-1}(a'),$
and so an element $x$ with $w(x)=0$ must be only $c_{n-1}(y)$.  Hence there can not exist $c_i(y)$ for $i<n-1$, and a contradiction.
Thus we see  $a''=a'$.

Since $c_{n-1}(y)\in Im(Q_0)$, it is a $p$-torsion in $CH^*(X)$,
and this means 
\[  pc_{n-1}(y)+v_1c_{n-2}'+...+v_{n-1}c_0'=0\quad in\  \Omega^*(X)/I_{\infty}^2\]
for some $c_i'\in \Omega^*(X)$.  In particular, we see that 
\[res_{\Omega}(v_{n-1}c_0')=-pv_{n-1}y
\in \Omega^*(\bar X)/(res_{\Omega}(I_{\infty}^2))\cong \Omega^*(\bar X)/I_{\infty}^3\] from
$res_{\Omega}(pc_{n-1}(y))=pv_{n-1}y$. (Note for $i>0$, $res_{\Omega}(c_i')=0$ $mod(I_{\infty}^2)$
by dimensional reason.)  So $res_{\Omega}(c_0')=-py$.
Thus we can take $c_{0}(y)=-c_0'$.  (Note that 
 $c_i(y)$ are only decided
in $\Omega^*(X)$ modulo $(Ker(res_{\Omega}),I_{\infty})$.)

Then from Lemma 6.3, there is $b\in H^{*,*'}(X;\bZ/p)$ such that
\[ Q_0(b)=c_{n-1}(y),\quad Q_{n-1}(b)=c_0(y).\]
From (3), this $b$ is determined uniquely,
that is
$b=Q_1...Q_{n-2}(a')$.  Hence we have
\[ c_0(y)=Q_{n-1}(b)=Q_{n-1}Q_{n-2}...Q_1(a').\]

We also get $c_i(y)=Q_0...\hat Q_i...Q_{n-1}(a')$ by the same arguments, using the following relation
\[ v_ic_{n-1}(y)=-v_{n-1}c_i(y) \quad in\ \Omega^*(X)/(I_{\infty}^2) \quad for\ 0<i<n-1.\] 
\end{proof}

Now we recall the norm variety.
Given   a pure symbol $a$ in the mod $p$  Milnor $K$-theory
$ K_{n+1}^M(k)/p$, by Rost,  we can
construct the norm variety $V_a$ such that 
\[ \pi_*([V_a])=v_n,\quad a|_{k(V_a)}=0\in K_{n+1}^M(k(V_a))/p\]
where $[V_a]=1\in \Omega^0(V_a)$,  $\pi:V_a\to pt.$ is the projection and $k(V_a)$ is the function field of $V_a$ over $k$.  
Note $I(V_a)=I_{n+1}$.
Rost and Voevodsky  showed that
there is $y\in CH^{b_n}(\bar V_a)$ for
$b_n=(p^{n}-1)/(p-1)$  such that $y^{p-1}$ is the fundamental class of $\bar V_a$,  that is
\[ \bZ_{(p)}\{y^{p-1}\}\cong CH^{p^n-1}(\bar V_a)\cong
CH_0(\bar V_a).\]  
Then $deg(V_a)=p$ and $y^{p-1}\not\in Res_{CH}$, but
$py^{p-1}\in Res_{CH}$.  Moreover the same facts hold for
$y^i$ with $1\le i\le p-1$.  Hence each $y^i$ satisfies
(1),(2) (and some modified (3)) 
 in Assumption 6.4, in particular $deg_v(y^i)>0$.

More strongly, Rost and Voevodsky show the following theorem.
\begin{thm} ([Ro1,2], [Vo1,4])
For a nonzero symbol $a\in K_{n+1}^M(k)/p$, 
let $M(V_a)$ be the motive of the norm variety $V_a$.
Then   there is an irreducible motive  $R_a$ (write
$R_{n}$ also) in $M(V_a)$ 
such that
$ CH^*(\bar R_a)\cong \bZ_{(p)}[y]/(y^p) $ and 
\[ CH^*(R_a)
\cong \bZ_{(p)}\oplus  M_n(y)'[y]/(y^{p-1}).\]
\end{thm}

We can also prove  
\begin{thm} ([Vi-Ya], [Ya6])
The map  $res_{\Omega}:
\Om(R_a)\to \Om(\bar R_a)$ is injective and 
\[Res_{\Omega}(X)=\Om\{1\}
\oplus I_n\{y,...,y^{p-1}\}\subset \Om[y]/(y^p)\cong \Omega^*(\bar R_n).\]
\end{thm}
\begin{cor}    For the  Rost motive $R_{n}$ and
$1\le i\le p-1$, we have  $deg_v(y^i)=n$ and
 $CH^*(R_n)/p\cong P_1(R_n).$\end{cor}

By Rost and Voevodsky, it is known that 
([Ro1,2], [Vo1-3], [Su-Jo])
\[ H^{*,*'}(\chi_a;\bZ/p)\cong H^{*,*'}(R_a;\bZ/p)\]
for all $ *\le 2b_n,$ $ *'\le *,$
and there is a short exact sequence [Or-Vi-Vo]
\[ 0\to H^{*+1,*}(\chi_a;\bZ/p)\stackrel{\tau}{\to}
 K_{*+1}^M(k)/p \to K_{*+1}^M(k(V_a))/p \]
where we identify $H^{*+1,*+1}(\chi_a;\bZ/p)\cong K_{*+1}^M(k)/p$.
Hence there exists $a'=\tau^{-1}(a)\in H^{*+1,*}(R_a;\bZ/p)$.

We can define $y\in CH^*(\bar R_a)$ by 
\[ y=p^{-1}res_{\Omega}( Q_1...Q_{n-1}(a'))\quad in\ 
gr \Omega^*(\bar R_n)\cong \Omega^*\otimes CH^*(\bar R_n).\]
We can prove Theorem 6.6 by using  
Theorem 6.5, 6.7, while to show Assumption 6.4  is rather difficult, and we still need Voevodsky' s original proof  ([Vo1,2,4]).

\section{Dickson invariant}

The contents in this section are not used in other sections
but arguments are very similar to those given in the preceding
section.  
Throughout this section, we assume that
$k$ contains $\xi_{p^2}$
a primitive $p^2$-root of the unity.
Since 
\[Q_0(\tau)=\xi_p=0\quad in\ k^*/(k^*)^p\cong H^{1,1}(pt.;\bZ/p),\]
we see $Q_i$ commutes with $\tau$, i.e., 
$  Q_i(\tau\cdot z)=\tau\cdot Q_i(z)$
for $z\in H^{*,*'}(X;\bZ/p)$.

We consider the case $X=B(\bZ/p)^{n}$
(while $\bar X$ is of course not cellular) with 
\[H^{*,*'}(X;\bZ/p)\cong K_*^M(k)\otimes \bZ/p[\tau,y_1,...,y_n]\otimes \Lambda(x_1,...,x_n)\]
for $p$ odd prime.  
Then $x=x_1...x_n$ is invariant under the $SL_n(\bZ/p)$-action.   So each
\[q_i(x)=Q_0...\hat Q_i...Q_n(x)\quad x=x_1...x_n\]
is $SL_n(\bZ/p)$- invariant.
Its invariant ring on the polynomial ring is well
known by Dickson
\[ \bZ/p[y_1,...,y_n]^{GL_{n}(\bZ/p)}\cong
\bZ[c_{n,0},....,c_{n,n-1}],\]
\[ \bZ/p[y_1,...,y_n]^{SL_{n}(\bZ/p)}\cong
\bZ[e_{n}, c_{n,1},....,c_{n,n-1}],\]
where $e_{n}^{p-1}=c_{n,0}$ and each $c_{n,i}$ is defined
by
\[ \Pi_{y\in \bZ/p\{y_1,...,y_n\}}(t+y)=t^{p^{n}}+\sum t^{p^{i}}c_{n,i}.\]
These $c_{n,i}$ is also represented by Chern classes.
Let $reg:(\bZ/p)^n\to U({p^n})$ be the (complex)
regular representation for the unitary group $U(p^n)$.  Then
\[ c_{n,i}=c_{p^n-p^i}(reg)=reg^*(c_{p^n-p^i}).\]

It is known by Mimura-Kameko [Mi-Kam]
\[ e_n c_{n,i}=Q_0...\hat Q_i...Q_n(x_1...x_n).\]
Hence the element $e_{n}c_{n,i}$ is just $q_i(x)$.
Hence we get 
\[q_i(x)=c_{n,i}e_{n}=reg^*(c_{p^{n}-1})^{1/(p-1)}
  reg^*(c_{p^{n}-p^i}).
\]

Each element    $ a\in H^{1,1}(X;\bZ/p)\cong H_{et}^1(X;\bZ/p)$
can be represented by $\bA^1$-homotopy map
$X\to B\bZ/p$.  So we can write
\[ a=a_1...a_n=i_{a}^*(x_1...x_n) \quad for\ i_a:X\to  (BZ/p)^n.\]
\begin{thm}  Let $p$ be odd.
For an element  $a=a_1...a_n\in H^{n,n}(X;\bZ/p)$,
we have
\[  q_i(a)=Q_0... \hat Q_i... Q_{n}(a)\]
\[   =i_a^*(e_{n}c_{n,i})=i^*_a((reg^*(c_{p^{n}-1}))^{1/(p-1)}
 reg^*(c_{p^{n}-p^{i}}))\]
in $H^{2*,*}(X;\bZ/p)$.  Moreover we have
$pq_i(x)=v_iq_0(x)$ in  $BP^*(X)/(I_{\infty}^2,v_j|j>i).$
\end{thm}

It is known ([Vo1,2]) that
\[H^{*,*'}(B(\bZ/2)^n;\bZ/2)\cong K_*^M(k)\otimes \bZ/2[\tau,y_1,...,y_n, x_1,...,x_n]/(Re.)\]
\[ where\ \ Re.=(x_i^2=\tau y_i \ |\  1\le i \le n).\]
Giving grading by multiplying $\tau$, we have
\[(*)\quad gr H^{*,*'}(B(\bZ/2)^n;\bZ/2)\cong K_*^M(k)\otimes \bZ/2[\tau,y_1,...,y_n]
\otimes \Lambda(x_1,...,x_n).\] 
The above theorem also holds for $p=2$, considering
this graded ring \\
$grH^{*,*'}(B(\bZ/2)^n;\bZ/2)$ instead of 
$H^{*,*'}(B(\bZ/2)^n;\bZ/2)$.

Here, however, we give an another argument for $p=2$,
which is quite similar to that in the previous section.   Define
$d_{n,i}\in\bZ/2[x_1,...,x_n]$
by\[ \Pi_{x\in \bZ/2\{x_1,...,x_n\}}(t+x)=t^{2^{n}}+\sum t^{2^{i}}d_{n,i}\]
so that $d_{n,i}^2=c_{n,i}$ and we have the isomorphism
\[ \bZ/2[x_1,...,x_n]^{GL_n(\bZ/2)}\cong \bZ/2[d_{n,0},...,d_{n,n-1}].\]
Moreover for the real regular representation
$reg_{\bR}:(\bZ/2)^n\to BO(2^n)$ for the special orthogonal group
$SO(2^n)$, we have
\[ \tau^{-2^{n-1}+2^{i-1}+1} reg_{\bR}^*(w_{2^{n}-2^{i}})=d_{n,i}\]
where we identify the element $d_{n,i}$ with a homogeneous
element of $w(d_{n,i})=2$. 
We know from Lemma 5.7 in [Kam-Te-Ya]
\[ d_{n,i} =Q_0...\hat Q_i...Q_{n-2}(x_1...x_n)\quad for
\ i\le n-2\]
where $d_{n,i}$ is considered  in $(*)$
 (see for details $\S 5$ in [Kam-Te-Ya]).
\begin{thm}
Let $a=a_1...a_n\in H^{n,n}(X;\bZ/2)$ and $\tau a'=a$.
Then we have
\[  \tau \cdot q_i(a')=\tau\cdot Q_0... \hat Q_i... Q_{n-2}(a')
=Q_0...\hat Q_i...Q_{n-2}(a)\]
\[   =i_a^*(d_{n,i})=
i^*_a( \tau^{-2^{n-1}+2^{i-1}+1}reg^*_{\bR}(w_{2^{n}-2^i}))
.\]
\end{thm} 
{\bf Remark.}
Note that when $a\in K_*^M(k)/2\cong H^{*,*}(pt.;\bZ/2)$ and $X=R_n$, the element $q_i(a')=c_i(y)$ in the previous 
section. However, the value of the above equation is zero (indeed $Q_ia=0$). Hence the above theorem
may be seen  an extension of the fact $\tau c_i(y)=0$.

{\bf Remark.}
Smirnov and Vishik constructed the (generalized)
Arason map $e_n:I^n/I^{n+1}\to K_*^M(k)/2$
where $I$ is the fundamental ideal of the Witt ring $W^*(k)$
generated by even dimensional quadratic forms.
They used the Nisnevich classifying space $BO(n)_{Nis}$.
In this paper we write the etale classifying space
by $BO(n)_{et}$ or simply $BO(n)$.  The motivic cohomology is given as 
(Theorem 3.1.1 in [Sm-Vi])
\[ H^{*,*'}(BO(n)_{Nis};\bZ/2)\cong H^{*,*'}(pt.:\bZ/2)\otimes \bZ/2[u_1,...,u_n] \]
where $u_i=\tau^{-[(i+1)/2]}w_i$ identifying
$w_i\in H^{i,i}(BO(n)_{et};\bZ/2)$.  These elements are called as subtle characteristic classes in the paper [Sm-Vi].

Let $X_q$ be the quadric for a quadratic form $q$, and 
$\chi_q=\chi_{X_q}$.  Smirnov-Vishik proved (Theorem 3.2.34, Remark 3.2.35 in [Sm-Vi]) that
when $q\in I^n$, there is a map $f_q: \chi_q\to BO(N)_{Nis}$
(with $N=dim(q)$), and the Arason map $e_n(q)$ can be  defined as (for $1\le i\le n-2$)
\[   e_n(q)=\tau (Q_0...\hat Q_i...Q_{n-2})^{-1}f_q^*(u_{2^{n}-2^i}).\]
We consider the following diagram (not assumed commutative) for $*''=2^{n }-2^i$
\[ \begin{CD}
H^{*,*'}(\chi_q;\bZ/2) @<{f_q^*}<<  H^{*,*'}(BO(N)_{Nis};\bZ/2)
@<{\tau^{-*''/2}j^*}<<   H^{*,*'}(BO(N)_{et};\bZ/2)\\
  @VVV  @.   @VVV\\
   H^{*,*'}(X_q:\bZ/2) @<{i_{a}^*}<< H^{*,*'}(B(\bZ/2)^n)_{et};\bZ/2)
       @<{\tau^{-*''/2+1}j^*reg^*}<< H^{*,*'}(BO(2^n)_{et};\bZ/2)
\end{CD}\]
where $N\ge max(dim(q), 2^n)$, 
and $j: BO(N)_{Nis}\to BO(N)_{et}$
are natural induced maps, and $i_a^*$ is the induced map
from $a=a_1...a_n\in H_{et}^n(X_q;\bZ/2)$.  Hence
we have  two elements
\[ u=f^*_q(u_{*''})=  f_q^* (\tau^{-*''/2}\cdot j^*(w_{*''}))
\]
\[v=i_{e_n(q)}^*d_{n,i}=i_{e_n(q)}^*
(\tau^{-*''/2+1}\cdot reg^*_{\bR}(w_{*''})).\]
Then we see $\tau u=v$, but it is only zero.
Here there is $x'\in H^{1,0}((B\bZ/2)_{Nis};\bZ/2)$
such that $\tau x'=x$.  hence $\tau^{-1}d_{n.i}$
exists in $H^*(B(\bZ/2)^n_{Nis};\bZ/2)$.
However we can not see a good map $X_q\to B(\bZ/2)^n_{Nis}$.

\section{quadrics and $p=2$}

In this section, we rewrite arguments of quadrics by
using $deg_v(y)$, $M_n(y)$ and $P_1(X)$.
Let $h\in CH^1(X)$ be the hyperplane section.
Let $X$ be a quadric of $dim(X)=2{\ell}-1$. (The case $dim(X)=2\ell$ is given with some modification of odd case.)  From Rost and Toda-Watanabe ([Ro1], [Tod-Wa])
\[ CH^*(\bar X)\cong \bZ_{(2)}[h,y]/(h^{\ell}=2y,y^2)
\] \[
\stackrel{  }{\cong}
\bZ_{(2)}\{1,h,...,h^{\ell-1}\}\oplus \bZ_{(2)}\{y,...,h^{\ell-1}y\}.\]
Hence $h\in Res_{\Omega}(X)$ and $deg_v(h)=0$.
Note that $CH^*(\bar X)\cong \bZ_{(2)}$ for each 
$0\le *\le 2\ell-1$.  Hence from Lemma 3.4,
\begin{lemma}
We have $P_1(X)\subset CH^*(X)/2.$
\end{lemma}

Throughout this section, we assume that $X$ is $anisotropic$.

This means
$deg_v(h^{\ell-1}y)\ge 1$ since $h^{\ell-1}y$ is the fundamental class of the $(2\ell-1)$-dimensional
manifold $\bar X$.  Thus we get
\[deg_v(h^i)=0,\quad deg_v(h^iy)\ge 1\quad for\ all \ 0\le i\le \ell-1.\]

\begin{lemma} Let $X$ be an (anisotropic)
quadric of $dim(X)=(2\ell-1)$.  Let us write $d_i=deg_v(h^iy)$.
Then
$ d_0\le d_1\le...\le d_{\ell-1}$ and 
\[ P_1(X)\cong \oplus_{0\le i\le \ell-1}( M_0(h^i)\oplus M_{d_i}(h^iy)).\]
\end{lemma}
\begin{proof}
If $v_nh^iy\in Res_{\Omega}(X)$, then so is $h\cdot v_nh^iy$ from $h\in Res_{\Omega}(X)$.  This implies
$d_i\le d_{i+1}$.
Since $J(h^iy)\cong I_{d_i}$ $ mod(I_{\infty}^2)$, we have $C(h^iy)
\cong M_{d_i}(h^iy)$.
\end{proof}

We recall that
\[ M_0(h^i)\oplus M_{d_i}(h^iy)\cong
\bZ/2\{h^i,c_0(h^iy)\}\oplus \bZ/2\{c_1(h^iy),...,c_{d_i-1}(h^iy)\}.\]
We note $c_0(h^iy)=2h^iy=h^{\ell+i}$.
Hence we can write
\[(*)\quad  P_1(X)\cong \bZ/2[h]/(h^{2\ell})\oplus \oplus _{0\le i\le \ell-1}\bZ/2\{c_1(h^iy),...,c_{d_i-1}(h^iy)\}.\]

Note that the Rost motive is isomorphic to
\[ CH^{*-k_i}(R_{d_i})/2\cong 
\bZ/2\{h^{k_i},c_0(h^iy)\}\oplus \bZ/2\{c_1(h^iy),...,c_{d_i-1}(h^iy)\}\]
where 
$ k_i=1/2(|h^iy|-|R_{d_i}|)=i+(2^{\ell}-2^{d_i}).$ 
Therefore we have
\begin{cor}
If the motive $M(X)$ is a sum of Rost motives
(e.g., an excellent quadric), then 
$CH^*(X)/2\cong P_1(X)$ as above. 
In particular
the set
$ \{k_0,...,k_{\ell-1}\}=\{0,1,...,\ell-1\}.$
\end{cor}
By using $(*)$ as above, we can prove
\begin{cor}  (Theorem 4.7 in [Ya3])
The Chow ring $CH^*(X)/2$ has the subring\[ P_1(X)\cong \bZ/2[h]/(h^{2\ell})\oplus_{j=0}^{s} \bZ/2[h]/(h^{\ell-f_j})\{u_j\}\]
where  $f_j$ is the smallest $f$ such that
$v_jh^fy\in Res_{\Omega}(X)$  (i.e., $deg_v(h^fy)=j+1$),  and $u_j=c_j(h^{f_j}y)=v_jh^{f_j}y$,
and so $s=deg_v(h^{\ell-1}y)-1$.
\end{cor}
\begin{proof}
 The result follows from changing order of sums
such that
\[ Res_{\Omega}(X)\supset \oplus I_{d_i}\{h^iy\}=\oplus_{i,j} (v_jh^iy)=\oplus h^{i-f_j}(v_jh^{f_j})=\oplus (h^{i-f_j}u_j).\]
(For details see the proof of Theorm 4.7 in [Ya3].)
\end{proof}

Let $a=a_0...a_n \not =0 \in K_{n+1}^M(k)/2$
and $\phi_{n+1}=\la \la a_0,..., a_n\ra \ra $ be 
 the $(n+1)$-th Pfister form associated to $a$.
Let $q'$ is the subform of codimension $1$,
that is, $q'$ is the maximal neighbor of the $(n+1)$-th Pfister form $\phi_{n+1}$.

{\bf Example 1.} ([Ya3]) 
Let  $X=X_{q'}$ be the maximal neighbor of the $(n+1)$ Pfister quadric $(\ell=2^{n}-1$).  Then $deg_v(yh^i)=n$ for all $0\le i\le \ell-1$.
Hence we have the ring isomorphism 
\[CH^*(X_{q'})/2\cong \bZ[h]/(h^{2^{n+1}-2})\oplus
\bZ/2[h]/(h^{2^n-1})\{u_1,...,u_{n-1}\}\]
where $u_i=c_i(y)=v_iy\in \Omega^*(\bar X)$ so that  $u_iu_j=0$.

{\bf Example 2.}
 ([Ya3])
Let $X_{q''}$ be 
 the minimal       neighbor of the $(n+1)$ Pfister quadric ($\ell=2^{n-1}$), namely, the norm variety.  Then $deg_v(vh^{\ell-1 })=n$ and $deg_v(yh^i)=n-1$ for all $0\le i\le \ell-2$. Hence we have the ring isomorphisms 
\[CH^*(X_{q''})/2\cong 
\bZ/2[h]/(h^{2^{n}})\oplus
\bZ/2[h]/(h^{2^{n-1}})
\{u_1',...,u_{n-2}'\}
\oplus \bZ/2\{u_{n-1}'\}\]
where  $u'_i=c_i(y)$ for $1\le i\le n-2$ and $u_{n-1}'=c_{n-1}(h^{2^{n-1}-1}y)$.

Let $f: X\to Y$ be a map of anisotropic quadrics.
Suppose $f^*(h_Y)\not =0 \in CH^2(\bar X)/2$
(i.e. $f^*(h_Y)=h_X$).  Let $y_Y$  be
the ring generator of $CH^*(\bar Y)$ such that
 $2y_Y=h^{\ell_Y}$ for $\ell=1/2(dim(Y)-1)$.
Similarly define the ring generator $y_X$ in $CH^*(\bar X)/2$.
Then 
$f^*(y_Y)=h^{\ell_Y-\ell_X}y_X$ (from $2y_X=h^{\ell_X}$), and  so 
$dim(X)\le dim(Y)$. Thus we see

(1)\quad $ f^*: CH^*(\bar Y)\to CH^*(\bar X)$ is injective
for $*\le dim(X)$.

(2)\quad  each $h^jy_X\not \in Res_{CH}$
for
$0\le j \le \ell_X-1$.

\noindent Recall $P_1^N(X)=\oplus_{|y_i|\le N}
M_{d_i}(y_i)$.  From Lemma 3.6,  we have
\begin{thm}
Let $f:X\to Y$ be a map of anisotropic quadrics
such that $f^*:CH^2(\bar Y)/2\to CH^2(\bar X)/2$ is
 non zero.
Then the induced map  
\[ P_1^{ dim(X)}(Y) \to P_1(X)\quad is\ injective.\]
\end{thm}

{\bf Example 3.}
Let $X_{min}$ (resp. $X_{max}$) be the minimal
(resp. maximal) neighbor of the $(n+1)$-Pfister
quadric and $e:X_{min}\subset X_{max}$ be the embedding.
Note $dim(X_{min})=2^n-1$ and we see
\[ P_1^{ 2^n-1}(X_{max})\cong \bZ[h]/(h^{2^n})\oplus \bZ/2\{
u_1,...,u_{n-1}\}\]
from Example 1.  Hence we see
that it injects into $P_1(X_{min})$, by  $e^*(u_i)=h^{2^{n-1}-1}u_i'$ for $i\le n-2$ and 
$e^*(u_{n-1})=u_{n-1}'$.

\begin{cor}
Let $Y$ be an anisotropic  quadric and $X=Y\cap H^{2d}$ for $2d$-dimensional plane $H^{2d}$, and 
$e:X\subset Y$ be the induced embedding.
Then the Gysin map $e_*:P_1(X)\to P_1(Y)$ is injective.
Let us write by $f_j(X)$ (resp. $f_j(Y)$)
the smallest number $f$ such that $v_jh^fy_X\in Res_{\Omega}(X)$ (resp. $v_jh^fy_Y\in Res_{\Omega}(Y)$).
Then we have
\[  f_j(Y)-d\le f_j(X)\le f_j(Y)+d.\]
\end{cor}
\begin{proof}
First note that $e^*(y_Y)=h^dy_X$.
Note  $e_*(1)=h^{2d}$ and $e_*(y_X)=h^dy_Y$.
Let $f=f_j(Y)$.  Then $v_jh^fy_Y\in Res_{\Omega}(Y)$ 
and hence 
\[ e^*(v_jh^{f}y_Y)=v_jh^{f+d}y_X\in Res_{\Omega}(X).\]
This means $f_j(X)\le f_j(Y)+d$. 
If $v_jh^{g}y_X\in Res_{\Omega}(X)$ then
$ e_*(v_jh^gy_X)=v_jh^{g+d}y_Y\in Res_{\Omega}(Y).$
This fact implies $f_j(Y)\le f_j(X)+d$.
\end{proof}

 \section{Lie groups $G$ and the  flag manifolds $G/T$}  

Let $Y$ be a simply connected $H$-space of finite type
(e.g.  compact Lie group).
By Borel, its mod $p$-cohomology is (for $p$ odd)
\[ H^*(Y;\bZ/p)\cong P(y)/p\otimes \Lambda(x_1,...,x_{\ell}),\quad P(y)=\otimes_i^s\bZ_{(p)}[y_i]/(y_i^{p^{r_i}})
\]
where $|y_i|$ is even and $|x_j|$ is odd.  When $p=2$, a graded ring $grH^*(Y;\bZ/2)$ is isomorphic to the
right  hand side ring, e.g. $x_j^2=y_{i_j}$ for some $y_{i_j}$.

For ease of arguments, we assume $p$ $odd$
in this section  (however the similar results also hold for $p$=2).
Let us write by $QH^*(-)$ the indecomposable
module
\[QH^*(Y;\bZ/p)=H^*(Y;\bZ/p)/(H^{+}(Y;\bZ/p))^2 \cong \bZ/p\{y_1,...,y_s,x_1,...,x_{\ell}\}.\]
Then it is known by Lin, Kane [Ka]
\begin{lemma} (J.Lin, $\S 35$ in [Ka])
Let $p$ be an odd prime and  $Y$ be a finite simply connected $H$-space.  Then
\[Q^{even}H^*(Y;\bZ/p)=\sum_{n\ge 0}\beta P^nQ^{odd}H^*(Y;\bZ/p).\]
\end{lemma}
For connected but non simply connected space, the above theorem holds for
$*\ge 4$.  Moreover
\begin{lemma} ($\S 36$ in [Ka])  By the same assumption as
the preceding lemma, we see
      \[ Q^{2n}H^*(Y;\bZ/p)=0\quad unless \quad
 n=p^k+...+\hat p^i+...+1.\]
\end{lemma}
For example, the possibility of  even degree of 
 ring generators  are
\[ (p+1),\ \  (p+1,p^2+1),\ \ (p^2+p+1, p^3+p+1,p^3+p^2+1),...\]
Moreover, it is known that each generator is combined  by reduced power operations.

The first degree type $|y|=2(p+1)$ appears for
$p=3$,  exceptional Lie groups $F_4,E_6,E_7$, and 
for $p=5$, $E_8$ (and for $p=2$, $G_2,F_4,E_6$).  The second type
appears for $p=3$, $E_8$.  However it seems that there is
 no example for other types for $p$ odd.

So we may assume the first or second types above. 
Let us say that (simply connected) simple Lie groups are of type (I) and (II) respectively.
When type (I),
there are $x_1,x_2,y$ in $H^*(Y;\bZ/p)$ such that 
\[ (*)\quad x_2= P^1(x_1),\quad  y=Q_1(x_1)=Q_0(x_2)\]
with $|x_1|=3,$ $|x_2|=2p+1$ and $|y|=2p+2$. 
The existence of the above $y$ also holds for $p=2$.
For type (II), there are more elements $x_3,x_4$ and $y'$ such that
\[ (**)\quad y'=P^p(y) =Q_2(x_1)=Q_1(x_3)=Q_0(x_4),\]
where $|y'|=2(p^2+1)$, $|x_3|=2p^2-2p+3$, $|x_4|=2p^2+1$.
For $p=2$, the other types appear but we see ;
\begin{lemma} Let $G$ be a simply connected 
Lie group such that $H^*(G;\bZ)$ has $p$-torsion.
Then there are  $x_1,x_2,y$ satisfying $(*)$.
\end{lemma}

Next we recall the cohomology of flag manifolds
$G/T$.
Let $T$ be the  maximal torus of a simply connected 
compact Lie group $G$  and $BT$ the classifying space of $T$.
We consider the fibering
\[ (9.1)\quad G\stackrel{\pi}{\to}G/T\stackrel{i}{\to}BT\]
and the induced spectral sequence 
\[ E_2^{*,*}=H^*(BT;H^*(G;\bZ/p)) \Longrightarrow H^*(G/T;\bZ/p).\]
The cohomology of the classifying space of the torus is  given by
\[H^*(BT)\cong S(t)=\bZ[t_1,...,t_{\ell}]\quad   with\ \ |t_i|=2.\]
where $\ell$ is also the number of the odd degree generators $x_i$ in
$H^*(G;\bZ/p)$.  It is known that $y_i$ are permanent cycles and 
that there is a regular sequence ([Tod],[Mi-Ni])
$(\bar b_1,...,\bar b_{\ell})$ in $H^*(BT)/(p)$ such that $d_{|x_i|+1}(x_i)=\bar b_i$.  Thus we get
\[ E_{\infty}^{*,*'}\cong grH^*(G/T;\bZ/p)\cong P(y)/p\otimes 
S(t)/(\bar b_1,...,\bar b_{\ell}).\]

 Moreover we know that $G/T$ is a manifold of torsion free, and 
\[(9.2)\quad H^*(G/T)_{(p)}\cong \bZ_{(p)}[y_1,..,y_k]\otimes S(t)/(f_1,...,f_k,b_1,...,b_{\ell})\]
where $b_i=\bar b_i\ mod(p)$ and $f_i=y_i^{p^{r_i}}\ mod(t_1,...,t_{\ell}).$
Since $H^*(G/T)$ is torsion free,  we also know 
\[(9.3)\quad BP^*(G/T)\cong BP^*[y_1,...,y_k]\otimes S(t)/(\tilde f_1,...,\tilde f_k,
\tilde b_1,...,\tilde b_{\ell})\]
where $\tilde b_i=b_i\ mod(BP^{<0})$ and $\tilde f_i=f_i\ mod(BP^{<0}).$

Here we will study a relation between $Q_i$-actions on
$H^*(G;\bZ/p)$ and $v_i$-module structure of $BP^*(G/T)$.
Recall that $k(n)^*(X)$ is the connected Morava K-theory
with the coefficients ring $k(n)^*\cong \bZ/p[v_n]$
and $\rho : k(n)^*(X)\to H^*(X;\bZ/p)$ is the natural
(Thom) map.  Recall that there is an exact sequence
(Sullivan exact sequence [Ra], [Ya2]) induced from
(the topological version of) (2.1)
\[...\to k(n)^{*+2(p^n-1)}(X)\stackrel{v_n}{\to} k(n)^*(X)\stackrel{\rho}{\to}
H^*(X;\bZ/p)\stackrel{\delta}{\to}...\]
Here it is known that $\rho\cdot \delta(x)=Q_n(x)$
the Milnor operation.

We consider  the Serre spectral sequence  
\[ E_2^{*,*'} \cong H^*(B;H^{*'}(F;\bZ/p))
\Longrightarrow  H^*(E;\bZ/p), \]
induced from the fibering $F\stackrel{i}{\to} E\stackrel{\pi}{\to}B$ with $H^*(B)\cong H^{even}(B)$.
\begin{lemma} (Lemma 4.3 in [Ya1])
In the spectral sequence $E_r^{*,*'}$ above,
suppose that  there is $x\in H^*(F;\bZ/p)$ such that
\[ (*)\quad  y=Q_n(x)\not=0 \quad  and \quad
b= d_{|x|+1}(x)\not = 0\in E_{|x|+1}^{*,0}.\]
Moreover suppose that  $E_{|x|+1}^{0,|x|}
\cong \bZ/p\{x\}\cong \bZ/p$.
Then there are $y'\in k(n)^*(E)$ and $b'\in k(n)^*(B)$ such that
$i^*(y')=y$, $\rho(b')=b$ and that  in $k(n)^*(E)$,
\[(**)\quad v_ny'=\lambda \pi^*(b')\quad  
\ \  for \ \lambda\not =0\in \bZ/p.\]
Conversely if $(**)$ holds in $k(n)^*(E)$ for $y=i^*(y')\not =0$ and $b=\rho(b')\not =0$, then there is $x\in H^*(F;\bZ/p)$ such that $(*)$ holds.
\end{lemma}
\begin{cor} 
In the spectral sequence converging to $H^*(G/T;\bZ/p)$,
let $b\not =0$ be the transgression image of $x$, i.e.
$d_{|x|+1}(x)=b$.  Then we have its lift $b\in BP^*(BT)$
such that in $ BP^*(G/T)/\II$  
\[ b=py(0)+v_1y(1)+...+v_iy(i)+...\]
where $y(i)\in H^*(G/T;\bZ/p)$ with $\pi^*y(i)=Q_ix$.
\end{cor}

\section{versal flag varieties}

 Let $G_k$ be the split reductive algebraic  group corresponding to $G$, and $T_k$ be the split maximal
torus corresponding to $T$.  Let $B_k$ be the Borel subgroup
with $T_k\subset B_k$.   Note that $G_k/B_k$ is cellular, and
$CH^*(G_k/T_k)\cong CH^*(G_k/B_k)$.
Hence we have   
$ CH^*(G_k/B_k)\cong H^{2*}(G/T)$
and $CH^*(BB_k)\cong H^{2*}(BT).$

Recall the algebraic cobordism $\Omega^*(X)$.
There is a natural (realization) map $\Omega^*(X)\to
BP^*(X(\bC))$.  In particular, we have 
$\Omega^*(G_k/B_k)\cong BP^*(G/T).$
Recall  $I_n=(p,v_1,...,v_{n-1})$ and we  also note
\[ \Omega^*(G_k/B_k)/I_{\infty}\cong
CH^*(G_k/B_k)/p\cong H^*(G/T)/p.\] 

Let $\bG$ be a nontrivial $G_k$-torsor.
We can construct a twisted form of $\bG_k/B_k$ 
by $(\bG\times G_k/B_k)/G_k\cong \bG/B_k.$
We will study  the twisted flag variety $\bF=\bG/B_k$.

Let us consider an embedding of $G_k$ into the general linear group $GL_N$ for some large  $N$.  This makes $GL_N$ a $G_k$-torsor over the quotient variety $S=GL_N/G_k$.
Let $F$ be the function field $k(S)$ and  define
the $versal$ $G_k$-$torsor$ $E$ to be the $G_k$-torsor over $F$ given by the generic fiber of $GL_N\to S$. 
(For details, see [Ga-Me-Se], [To2], [Me-Ne-Za], [Ka1].)

The corresponding flag variety $E/B_k$ is called the 
$versal$ flag
variety, which is considered as the most complicated twisted
flag variety (for given $G_k$). 
In fact, for each $G_K$-torsor $\bG'$ over 
an extension $K$ over $k$, we have the natural
(specialization) map
$CH^*(E/B_k)\to CH^*(\bG'/B_k)$. 
 In particular,  the Chow ring
$CH^*(E/B_k)$ is not dependent to the choice
of  generic $G_k$-torsors $E$ (Remark 2.3 in [Ka1]). 
 
{\bf Note.}  Hereafter in this paper, we $always$ $assume$  that
$\bG$ (and hence $\bF$) means $versal$.
Hence $CH^*(\bF)$ means the Chow ring defined over the field $k(S)$ above (but not over $k$).

Karpenko proves the following result
 for a versal flag variety. 
\begin{thm}
(Karpenko Lemma 2.1 in [Ka1], [Me-Ne-Za])
Let $h^*(-)$ be an oriented cohomology theory
e.g., $h^*(X)=CH^*(C), \Omega^*(X)$.
Then the natural map
$h^*(BB_k)\to h^*(\bG/B_k)$ is surjective.
\end{thm}
\begin{cor}
The cohomology  $h^*(\bF)=h^*(\bG/B_k)$
is multiplicatively generated by elements $t_i$ in $S(t)$. 
\end{cor}

The versal case  of the main result in Petrov-Semenov-Zainoulline \cite{Pe-Se-Za} is
given  as
\begin{thm} (Theorem 5.13 in [Pe-Se-Za])
There is a $p$-localized  motive $R(\bG)$ such that
the ($p$-localized) motive $M(\bF)$ of the variety $\bF$ is decomposed
\[M(\bF)_{(p)}\cong  R(\bG)\otimes T(\bG)\quad with\ T(\bG)=(\oplus _u
 \bT^{\otimes u}).\]
Here $\bT^{\otimes u}$ are Tate motives with 
$CH^*(T(\bG))/p\cong  S(t)/(p,b)=S(t)/(p,b_1,...,b_{\ell}).$   Hence, we have additively
\[  CH^*(\bF)/p\cong 
CH^*( R(\bG))\otimes S(t)/(p,b).\]
For  $\bar R(
\bG)=R(\bG)\otimes \bar k$, we have 
$ CH^*(\bar R(\bG))/p\cong P(y)/p$.
\end{thm}

Hence we have surjections for the  variety $\bF$
\[ CH^*(BB_k)\twoheadrightarrow  CH^*(\bF)\stackrel{pr.}{\twoheadrightarrow} CH^*(R(\bG)).\]
We study in  \cite{YaC}
what elements in $CH^*(BB_k)$ 
generate $CH^*(R(\bG))$.

For ease of notations, let us write
$A(b)=\bZ/p[b_1,...,b_{\ell}]$.
By giving the filtration on $S(t)$ by $b_i$, we 
can write 
$gr S(t)/p\cong A(b)\otimes S(t)/(b).$
In particular, we have maps
$ A(b)\stackrel{i_A}{\to} CH^*(\bF)/p\to CH^*(R(\bG))/p.$
We also see that
the above composition map is surjective
(see also Lemma 10.5 below).

\begin{lemma}
Let $pr:CH^*(\bF)/p\to CH^*(R(\bG))/p$, and 
 $0\not =x\in Ker(pr)$.  Then  
$ x=\sum b't'$ with $b'\in A(b),$ $0\not =t'\in S(t)^+/(p,b)$
 i.e., $ |t'|>0.$
\end{lemma}
Let us write 
\[  y_{top}=\Pi_{i=1}^s y_i^{p^{r_i}-1}\quad (resp.\ t_{top})\]
the generator of the highest  degree 
in $P(y)$ (resp. $S(t)/(b)$) so that $f=y_{top}t_{top}$
is the fundamental class in $H^{2d}(G/T)$
for $2d=dim_{\bR}(G/T)$.
For $N>0$, let us  write \ \ 
\[A(b)_N=\bZ/p\{b_{i_1}...b_{i_k}| |b_{i_1}|+...+|b_{i_k}|\le N\}
\subset A(b).  \]
\begin{lemma}  For $N=|y_{top}|$, the map
$ A(b)_N\to CH^*(R(\bG))/p$ is surjective.
\end{lemma}
\begin{proof}  In the preceding lemma,  $A_{N}\otimes t'$ for $|t'|>0$ maps zero in $CH^*(R(\bG))/p$.  Since each element
in $S(t)$ is written by an element in $A_N\otimes S(t)/(b)$,
we have the lemma.
\end{proof}
\begin{cor}
If  $b_i\not=0$ in $CH^*(X)/p$,  then  so in $CH^*(R(\bG_k))/p$.
\end{cor}
\begin{proof}
Let $pr(b_i)=0$.  From Lemma 10.4, 
 $b_i=\sum b't'$ for $|t'|>0$, and hence $b'\in Ideal(b_1,...,b_{i-1})$.
This contradict to that $(b_1,...,b_{\ell})$ is regular.
\end{proof}

Now we consider the torsion index $t(G)$.
Let $dim_{\bR}(G/T)=2d$.  Then the torsion index is defined as
\[ t(G)=|H^{2d}(G/T;\bZ)/i^*H^{2d}(BT;\bZ)|
\quad where\ i:G/T\to BT.\]
Let $n(\bG)$ be the greatest common divisor of the degrees of all finite field extension $k'$ of $k$ such that $\bG$ 
becomes trivial over $k'$.  Then by Grothendieck 
\cite{Gr}, it is known that $n(\bG)$ divides $t(G)$.  Moreover, if  $\bG$ is versal, then $n(\bG)=t(G)$   (\cite{To2}, \cite{Me-Ne-Za}, \cite{Ka1}), which implies
that each element in $t(G)P(y)$ is represented by element in $ CH^*(BB_k)$.

It is well known that  if $H^*(G)$ has a $p$-torsion, then
$p$ divides the torsion index $t(G)$. Torsion index for
simply connected compact Lie groups are completely determined by Totaro \cite{To1}, \cite{To2}.

\begin{lemma} (\cite{YaC})
Let $\tilde b=b_{i_1}... b_{i_k}$ in
$S(t)$  such that
in $H^*(G/T)$
\[ \tilde b=p^s(y_{top}+\sum yt),\quad
|t|>0\]
for some $y\in P(y)$  and $t\in S(t)^+$.
Then $t(G)_{(p)}\le p^s$. 
\end{lemma}

From Lemma 9.5,  we have
\begin{lemma}  Suppose that $deg_v(y)>0$ for
$y\in P(y)/p$ and $d_{|x_j|+1}(x_j)\not =0$.
(Hence $d_k(x_j)=0$ for all $k\le |x_j|$.)  Then the degree $deg_{v}(y)$ is the max of the number $n+1$ such that in $H^*(G;\bZ/p)$
\[Q_n(x_j)=y  \quad  but \quad 
Q_k(x_j)=0\ \ for \ 0\le k< n.\]
\end{lemma} 
\begin{proof}  
From  Lemma 9.5, we have, in $\Omega^*(\bar R(\bG))/(\II)$,
\[ d_{|x_j|+1}(x_j)=b_j=py(1)+...+v_{n-1}y(n-1)+v_ny(n)+...\]
with $\pi^*(y(i))=Q_i(x_j)$. 

Suppose that the condition for $Q_i$ is satisfied.
For $k<n$, we assumed $Q_k(x)=0$. Since 
$\pi^*(y(k))=0$, we see $y(k)=\sum yb$ for $y\in P(y)$ and $b\in A(b)^+$.
So $y(k)\in I_{\infty}k(n)^*(G/T)$ and $ v_ky(k)=0$ in $mod(\II)$.
Here we have $\pi^*(y(n))=y$, and hence $b_j$ is written 
\[(*)\quad v_ny+v_{n+ 1}y(n+1)+...=b\in Res_{\Omega}(X)
.\]
Hence $deg_v(y)\ge n+1$.

Conversely, suppose $(*)$ with $mod(I_{\infty}^2)$.
Here we note that we can write $b=b_k$
since $b_i\in I_{\infty}k(n)^*(G/T)$.
Then
\[  v_n y=b\quad mod(I_{\infty}^2)=mod(v_n^2)\ \ 
 in \  k(n)^*(G/T).\]
 So there is  $y'\in k(n)^*(G)$ with $v_ny'=0$ and  $y=y'\ mod(v_n)$.
By the Sullivan
exact sequence, there is $x\in H^*(G;\bZ/p)$ such that
$Q_n(x)=y'=y$ and $d_{|x|+1}(x)=b$.
\end{proof}

\section{The orthogonal  group $SO(2\ell+1)$ and $p=2$}

At first we consider the special 
orthogonal groups $G=SO(m)$ and $p=2$
, while it is not simply connected.
The $mod(2)$-cohomology is written as ( see for example \cite{Mi-Tod}, \cite{Ni})
\[ grH^*(SO(m);\bZ/2)\cong \Lambda(x_1,x_2,...,x_{m-1}) \]
where $|x_i|=i$, and the multiplications are given by $x_s^2=x_{2s}$.

For ease of argument,  we only consider the case
$m=2\ell+1$.
Hereafter this section, we assume $(G,p)=(SO(2\ell+1),2)$.
\[ H^*(G;\bZ/2)\cong P(y)\otimes \Lambda(x_1,x_3,...,x_{2\ell-1}) \]
\[ grP(y)/2\cong \Lambda(y_2,...,y_{2\ell}), \quad 
letting\ y_{2i}=x_{2i}\ \ (hence \ y_{4i}=y_{2i}^2).\]

The Steenrod operation is given as 
$Sq^k(x_i)= {i\choose k}(x_{i+k}).$
The $Q_i$-operations are given by Nishimoto \cite{Ni}
\[Q_nx_{2i-1}=y_{2i+2^{n+1}-2},\qquad Q_ny_{2i}=0.\]
It is well known that   the transgression 
$b_i=d_{2i}(x_{2i-1})=c_i$ is the  $i$-th elementary symmetric function
on $S(t)$. 
 Moreover we see that 
$Q_0(x_{2i-1})=y_{2i}$ in $H^*(G;\bZ/2)$.
From Corollary 8.3, we have
\begin{cor} In $BP^*(G/T)/\II$, we have  
\[c_i= 2y_{2i}+\sum_{n\ge 1} v_ny(2i+2^{n+1}-2) \]
for some $y(j)$ with $\pi^*(y(j))=y_{j}$.
\end{cor}

We have $c_i^2=0$ in $CH^*(\bF)/2$ from
the natural inclusion $SO(2\ell+1)\to Sp(2\ell+1)$
(see \cite{Pe}, \cite{YaC}) for the symplectic group
$Sp(2\ell+1)$.

Let us write by $\Lambda_{\bZ}(a_1,...,a_s)$ the $\bZ_{(p)}$
-free module such that  \\
$\Lambda_{\bZ}(a_1,...,a_s)/p\cong
\Lambda(a_1,...,a_s).$ 
Then we have 
\begin{thm} (\cite{Pe}, \cite{YaC}) 
Let $(G,p)=(SO(2\ell+1),2)$.
Then
\[ CH^*(R(\bG)) \cong 
\Lambda_{\bZ}(c_1,...,c_{\ell}).\]
\end{thm}

\begin{cor} We have
\[ J(y_{2i_1}...y_{2i_r})=(2^r)\subset BP^*
\quad for \ 1\le i_1<...<i_r \le \ell.\] 
Hence $CH^*(R(\bG))/2\cong P_{\ell}(R(\bG))$,
and  $P_1(R(\bG))\cong\bZ/2\{1,c_1,...,c_{\ell}\}$.
\end{cor} 
\begin{proof}
From the preceding theorem, we see
\[  \oplus_{(i_1<...<i_r)} \bZ/2\{c_{i_1}...c_{i_r}\}
\cong CH^*(R(\bG))/2.
\]

On the other hand,
since $2y_{2i}=c_i\ mod(I_{\infty}^2)$, we have
\[2^ry_{2i_1}...y_{2i_r}=c_{i_1}...c_{i_r}\quad 
mod(I_{\infty}^{r+1}) \]
which is in $Res_{\Omega}(X)$.
Hence
$C(y_{2i_1}...y_{2i_r})= \bZ/2\{c_{i_1}...c_{i_s}\}.$

So we have 
$P_{\ell}(R(\bG))\cong \oplus_{(i_1<...<i_r)} C(y_{2i_1}...y_{2i_r})\cong CH^*(R(\bG))/2.
$
\end{proof}
The above example may not be so interesting since 
\[ deg_{v}(y_{2i})=1,\quad and 
\quad deg_{v}(y_{2i_1}...y_{2i_r})=-1\ for \ r\ge2.\]
However we consider an example
$deg_{v}(y)=n$ for each $n$, in an extension over 
the field $k$,
We study $CH^*(X|_K)/2$ for some interesting 
extension $K$ over $k$.
Let $K$ be an extension of $k$ such that
$X$ does not split over $K$ but splits over an extension
over $K$ of degree $2a$, $(a,2)=1$.  Suppose that
\[ (*)\ \  y_{2\ell}\not \in {Res}_{\Omega}(X|_K) \quad but \ \ 
 y_{2i}\in {Res}_{\Omega}(X|_K)\  \ for\ 1\le i\le \ell-1.\]
\begin{lemma}  Suppose $(*)$.  Then $\ell=2^n-1$
for  $n>0$.
\end{lemma}
\begin{proof}
We can  see that if $\ell\not =2^n-1$, then each $y_{2\ell}$
is a target of the Steenrod operation $Sq^{2k}$
of a sum of products of $y_{2i}$ for $i<\ell$. 
\end{proof}
\begin{lemma}  Suppose $(*)$  and $\ell=2^n-1$.
Then 
$v_{n-1}y_{2\ell}\in Res_{\Omega}(X|_K) $, i.e.,
$deg_{v}(y_{2\ell})=n$ for $X|_K$.
\end{lemma}
\begin{proof}
From Corollary 11.1,  we see
\[c_{\ell-2^j+1}=2y(2(\ell-2^j+2^0))+
v_1y(2(\ell-2^{j}+2^1))+...+v_j(y(2\ell))\]
\[ =v_j(y_{2\ell}) \mod(y_{2},y_{4},...,y_{2\ell-2}).\]
Hence we have
$ \rm{res}_{\Omega}\it (c_{\ell-(2^j-1)})=v_j(y_{2\ell})\  
 \mod(y_2,y_4,...,y_{2\ell-2}).$
\end{proof}
Thus we have
\begin{cor}  Suppose $(*)$ and $\ell=2^n-1$.
Then we have the injection  
\[ \Lambda(y_2,...,y_{2\ell-2})
\otimes
(M_0(1)\oplus M_n(y_{2\ell})) \subset CH^*(R(\bG)|_K)/2.\]
\end{cor}

At last of this section, we consider the case $X(\bC)=G/P$
with
\[ G=SO(2\ell+1)\quad and \quad P=U(\ell).\]
Let us write this $X$ by $Y$, i.e. $Y=\bG/P_k$.
From the fibering $SO(2\ell
+1)\to Y(\bC)\to BU(\ell)$,  we have the spectral sequence   
\[E_2^{*,*'}\cong H^*(SO(2\ell+1);\bZ/2)\otimes H^{*'}(BU(\ell))\]
\[\cong  P(y)\otimes \Lambda(x_1,...,x_{2\ell-1})\otimes \bZ/2[c_1,...,c_{\ell}]
\Longrightarrow H^*(Y(\bC);\bZ/2).\]
Here the differential
is given as $d_{2i}(x_{2i-1})=c_i$.  Hence
\[  CH^*(\bar Y)/2\cong H^*(Y(\bC);\bZ/2)\cong   P(y)/2.\]
This case is studied by Vishik [Vi] and Petrov [Pe]
as maximal orthogonal (or quadratic) grassmannian.
(see Theorem 5.1 in [Vi]).
From Theorem 11.2,  we have 
\begin{thm}  ([Vi],[Pe])  Let 
 $Y=\bG_k/U(\ell)_k$.
Then
\[ CH^*(Y)/2\cong  CH^*(R(\bG))/2\cong \Lambda(c_1,...,c_{\ell}).\]
\end{thm}

In [Vi], 
Vishik originally defined the $J$-invariant $J(q)$
 of a quadratic
form $q$ which corresponds to the quadratic grassmannian
(see Definition 5.11, Corollary 5.10 in [Vi])
by 
\[J(q)=\{i_k|y_{2i_k}\in Res_{\Omega}(X)\} 
=\{i_k|deg_v(y_{i_k})=0\} \subset \{0,...,\ell\}.\]
Let $I$ be the fundamental ideal of the Witt ring $W(k)$ 
so that $grW(k)=\oplus_n I^n/I^{n+1}\cong K_*^M(k)/2$
where $K_*^M(k)$ is the Milnor $K$-theory of $k$.
Smirnov and Vishik (Proposition 3.2.31 in [Sm-Vi]) prove that
\[ q\in I^n \quad\rm{ if\ and\ only\ if}\quad  \{0,...,2^{n-1}-2\}\subset \it J(q). \]
Hence the condition $(*)$ in Coeollary 11.6
is equivalent to $q\in I^n$ for the quadratic form $q$
(with $\ell=2^n-1$) corresponding to $Y|_K$. 
\begin{thm}
Let $q\in I^n$ be the quadratic corresponding $Y|_K$.
Then  there are  $m\ge n$ and the 
 injection for $\ell=2^m-1$ such that
\[ \Lambda(y_2,...,y_{2\ell-2})\otimes(M_0(1)\oplus M_m(y_{2\ell})) \subset CH^*(R(\bG)|_K)/2.\]
\end{thm}
\begin{cor}
Let $X_q$ be the quadric for an anisotropic  quadratic form $q$, and
$f\in CH^{dim(X_q)}(\bar X_q)$ be  a fundamental class.
If $q\in I^n$, then $deg_v(f)\ge n$.
\end{cor} 
\begin{proof}
For each generator $0\not =y\in CH^*(X_q)/2$, we see $deg_v(y)\le deg_v(f)$ from Lemma 8.1.
The result follows from $m=deg_v(y_{2\ell})\le deg_v(f)$.
\end{proof}
.

\section{$Spin(2\ell+1)$ and $p=2$}

Throughout this section, let $p=2$,  $G=SO(2\ell+1)$
and $G'=Spin(2\ell+1)$.
It is well known that
$ G/T\cong G'/T'$ for the maximal torus $T'$ of the spin group.
  By definition, we have the 
$2$ covering $\pi:G'\to G$.
It is well known that 
$\pi^*:   H^*(G/T)\cong H^*(G'/T')$.

Let $2^t\le \ell < 2^{t+1}$, i.e. $t=[log_2\ell]$.
The mod $2$ cohomology is
\[ H^*(G';\bZ/2)\cong 
 \cong P(y)'\otimes \Lambda(x_3,x_5,...,x_{2\ell-1})\otimes \Lambda(z),\quad |z|=2^{t+2}-1\]
where
$P(y)'\cong P(y)/(y_2)$.
(Here $z$ is defined by $d_{2^{t+2}}(z)=y^{2^{t+1}}$ for $0\not =y\in H^2(B\bZ/2;\bZ/2)$ in the spectral sequence induced from
the fibering $G'\to G\to B\bZ/2$.)
Hence
\[ grP(y)'\cong \otimes _{2i\not =2^j}\Lambda(y_{2i})\cong
\Lambda(y_6,y_{10},y_{12},...,y_{2\bar \ell})\]
where $\bar \ell=\ell-1$ if $\ell=2^j$ for some $j$, and
$\bar \ell=\ell$ otherwise.

The $Q_i$ operation for $z$ is given also 
by Nishimoto 
\cite{Ni}
\[ Q_0(z)=\sum _{i+j=2^{t+1},i<j}y_{2i}y_{2j}, \quad 
 Q_n(z)=\sum _{i+j=2^{t+1}+2^{n+1}-2,i<j}y_{2i}y_{2j}\ \ for\ n\ge 1.\]

We know that 
\[ grH^*(G'/T')/2\cong P(y)'\otimes S(t')/(2,c_2',.....,c_{\ell}',c_1^{2^{t+1}}).\]
Here $c_i'=\pi^*(c_i)$ and $d_{2^{t+2}}(z)=c_1^{2^{t+1}}$ in the spectral sequence
converging $H^*(G'/T')$.

Take $k$ such that $\bG$ is a versal $G_k$-torsor so that
$\bG'_k$ is also a versal $G_k'$-torsor.  Let us write
$\bF=\bG/B_k$ and $\bF'=\bG'/B_k'$.  Then
\[ CH^*(\bar R(\bG'))/2\cong P(y)'/2,\quad and \quad
    CH^*(\bar R(\bG))/2\cong P(y)/2.\]

The Chow ring $CH^*(R(\bG'))/2$ is not computed yet
for  $\ell\ge 6$,
while we have the following lemmas.
\begin{lemma}
Let $2^t\le \ell<2^{t+1}$.   
Then there  is a surjection
\[ \Lambda(c_2',...c_{\bar \ell}')\otimes \bZ/2[c_1^{2^{t+1}}]
\twoheadrightarrow CH^*(R(\bG')) /2.\]
\end{lemma}

\begin{lemma} We have
\[P_1(R(\bG'))\cong \begin{cases}
\bZ/2\{1,c_2',c_3'\} \quad for \ \ell=3,4 \\
\bZ/2\{1,c_2',...,c_{\bar \ell}', c_1^{2^{t+1}}\}\quad for\ \ell\ge 5.
\end{cases}\]
\end{lemma}
\begin{proof}
By Nishimoto, for $\ell>4$, $Q_0(z)=y\not =0 \in P(y)$
and $2y=c_1^{2^{t+1}}$ is an $\Omega^*$-module generator of $\Res_{\Omega}(X)$. So it is nonzero in $P_1(R(\bG))$.
For $\ell=3,4$, we see $Q_0(z)=0$ and 
$c_1^{2^{t+1}}=0$ in $P_1(R(\bG))$.

Let $\ell>4$.  We note $4y_{2i}y_{2j}, 2v_1y_{2i}y_{2^k+2}$
are in $Res_{\Omega}(X)$,  Since  $4,2v_1=0$ $mod(I_{\infty}^2)$,  the element $C(y_{2i}y_{2j})$ does not appear
in $P_1(R(\bG))$.
By the $Q_i$-actions, we can see,
as a submodule  of $gr\Omega^*(\bar (\bG))/(I_{\infty}^2),$
\[Res_{\Omega}(X)/I_{\infty}^2\cong (2)\{y_{2i}|i\not =2^j,2^j+1\}
\oplus (2,v_1)\{y_{2^j+2}\}\oplus (2)\{Q_0(z)\}\]
 In fact if $v_2y_{2i}\in Res_{\Omega}(X)$, then there is $x\in H^*(G;\bZ/2)$ such that
$Q_0(x)=Q_1(x)=0$ but $Q_2(x)=y_{2i}$.
But this is not the case from the $Q_i$-action by Nishimoto.
The restrictions are given
\[c_i' \mapsto 2y_{2i}, \  
 for \ i\not =2^j, \ \ 
c_{2^i}' \mapsto v_1y_{2^j+2}, \ \   
c_1^{2^t}\mapsto 2Q_0(z).\]
These imply  the lemma.
\end{proof}

\section{$Spin(7), Spin(9), Spin(11)$}

From  this section to the next  section,
we only consider $G=Spin(2\ell+1)$.
The space we consider $Res_{\Omega}(X)$ is  $X=R(\bG)$ the generalized Rost motive.

Hereafter in this paper, we change notations as follows.
 Let us write the element $c_1^j$  in the preceding section, by $e_j$,
and take off the dash of $c_i'$, i.e.
\[ e_{2^{t+1}}=c_1^{2^{t+1}},\quad c_2=c_2',\ c_3=c_3',\ ....,\ c_{\ell}=c_{\ell}'.\]

Now we consider examples. 
The groups $Spin(7),$ $Spin(9)$ are type $(I)$ of the 
 rank $\ell=3,4$ respectively.
Hence $R(\bG)\cong R_2$ the original Rost motive.
 The cohomology is given
\[ grH^*(G;\bZ/2))\cong\begin{cases} 
\bZ/2[y_6]/(y_6^2)\otimes \Lambda(x_3,x_5,z_7)\quad G=Spin(7),\ \\
\bZ/2[y_6]/(y_6^2)\otimes \Lambda(x_3,x_5,x_7,z_{15})
\quad G=Spin(9).\end{cases}\]
The cohomology operations are $Q_1(x_3)=Q_0(x_5)=y_6$.
Hence $P(y)\cong \bZ_{(2)}\{1,y_6\}$,
and $res_{\Omega}(c_2)=v_1y_6, \ res_{\Omega}(c_3)=2
y_6$.

\begin{thm} (\cite{YaC})  Let $G=Spin(7)$ or $Spin(9)$.
  Then\[Res_{\Omega}(R(\bG))\cong BP^*\{1\}\oplus(2,v_1)\{y_6\},\]  \[  CH^*(R(\bG))/2\cong P_1(R(\bG))\cong
\bZ/2\{1,c_2,c_3\}.\]
\end{thm}

Next, we consider $G=Spin(11)$.
The cohomology is written as 
\[H^*(G;\bZ/2)\cong \bZ/2[y_6,y_{10}]/(y_6^2,y_{10}^2)\otimes \Lambda(x_3,x_5,x_7,x_9,z_{15}).\]
By Nishimoto, we know $Q_0(z_{15})=y_6y_{10}$. It implies
$2y_6y_{10}=d_{16}(z_{15})=e_8$.
Since $y_{top}=y_6y_{10}$, we have $t(G)=2$.

Let us write  $k^*(X)=A\tilde k(1)^{2*,*}(X)$.
(It is written in $\S 5$ as $h^*(X)$ for $n=1$.
We use Lemma 5.4 for $n=1$.)
From Corollary 2.1, we can take in $k^*(\bar \bF)$
such that $c_2=v_1y_6 ,c_4=v_1y_{10}$.  Moreover
in $k^*(\bar \bF)/(I_{\infty}^2)$ we have
$c_3=2y_6,c_5=2y_{10}$.

\begin{thm} (Karpenko \cite{Ka2}, \cite{YaC}, \cite{YaG})
 Let $G=Spin(11)$. Then we have 
\[CH^*(R(\bG))/2 \cong gr_{geo}(R(\bG))/2\cong \bZ/2\{1, c_2,c_3,c_4,c_5,c_2c_4, e_8\},\]
where $c_2c_4=pr(c_2\cdot c_4)$ for the cup product $\cdot$ in $gr_{geo}^*(\bF)$.
\end{thm}
\begin{cor}  Let $G=Spin(11)$. 
Then the restriction map
$ res_{\Omega}$ is injective 
\[Res_{\Omega}(R(\bG))=BP^*\{1\}\oplus (2,v_1)\{y_6,y_{10}\}
\oplus (2,v_1^2)\{y_6y_{10}\}.
\]
That means
$CH^*(R(\bG))/2\cong P_2(R(\bG))$.
\end{cor}
\begin{proof} 
The right hand side module is contained in $Res_{\Omega}(X)$.  Its number of the $BP^*$-module generators 
is just $7=rank_2(CH^*(R(\bG))/2)$.
Hence it is  $Res_{\Omega}(X)$. 
The fact $P_2(R(\bG))\subset CH^*(R(\bG))/2$
also follows from Lemma 3.4, 4.5.  In fact 
there is no $\bZ_{(2)}
$-module generator $y(I)$ with $|y(I)|=|y_6y_{10}|-2$.
\end{proof}

We proved the above corollary from the preceding theorem.  However, we see the theorem from
the above corollary, as stated in the introduction.

\section{ $Spin(13)$.}

Throughout this section. let
$G=Spin(13)$ so that $\ell=6$.
 Then we have
\[ grP(y)\cong \Lambda(y_6,y_{10},y_{12}).\]
Hence $y_{top}=y_6y_{10}y_{12}$.
Since  
$2^t\le \ell<2^{t+1}$, so $t=2$.
Hence 
$e_8$ exists in $CH^*(R(\bG))$.  Note $e_8^2=pr(e_8\cdot e_8)=0$ in $CH^*(R(\bG))$
since $|e_8^2|>|y_{top}|$.
We also note $t(G)=2^2=4$ by Totaro, 
in fact $e_8c_6=4y_{top}$.
 
{\bf Note.}   From Corollary 5.5 in \cite{YaS}, the projection $pr(c_{i_1}...c_{i_s})$ is 
uniquely determined in $ k^*(\bar R(\bG))/(I_{\infty}^{s+1})\cong k^*\otimes P(y)/(I_{\infty}^{s+1})$.

We take $y_6,y_{10},y_{12}$ such that
\[ c_2=v_1y_6,\ \ c_4=v_1y_{10}, \ \ y_{12}=y_6^2\quad in\
k^*(\bar \bF).\]
(See \cite{YaS} for details.)
Recall that  the invariant ideal $I_{\infty}$
in $k^*$-theory is 
$I_{\infty}=(2,v_1)\subset k^*$.
Note that $pr(I_{\infty})\subset I_{\infty}$.
Then we can take $c_i\in k^*(R(\bG))$ such that
in $k^*(\bar R(\bG))/(I_{\infty}^3)$ 
\[c_3=2y_6, \ \ 
 c_5=2y_{10}+v_1y_{12},\ \ 
 c_{6}=2y_{12},\]
\[ e_8=2y_6y_{10}+v_1y_6y_{12}.\]
Here we used $P(y)^*=0$ for $*=8,14,20$,
and $ Q_1x_9=y_{12},\ Q_1z_{15}=y_6y_{12}$.

\begin{thm} (\cite{YaS}) 
We have 
\[ gr_{geo}(R(\bG))
\cong P(b)/(c_2c_5,c_2c_4c_5)\oplus
\bZ/2\{c_4c_3,e_8c_4,c_2c_6,c_4c_6\}\]
\[\oplus \bZ_{(2)}\{1,c_3,c_6,e_8. c_3c_6,c_5c_6, e_8c_6\}\]
where $P(b)=\Lambda_{\bZ}(c_2,c_4,c_5)/(2c_2,2c_4)$
\end{thm}
\begin{cor} We have the additive isomorphism
\[ gr_{geo}(R(\bG))\cong 
    A\otimes (\bZ_{(2)}\{1,c_6\}\oplus \bZ/2\{c_4\})\oplus
\bZ/2\{c_6c_4\}\]
where $A= \bZ_{(2)}\{1,c_3,c_5,e_8\}\oplus \bZ/2\{c_2\}.$
\end{cor}

\begin{lemma} (\cite{YaS})
The restriction image $Res_{\Omega}(X)\otimes_{\Omega^*} k^*$
for $X=R(\bG)$
is given by
\[Res_{\Omega}(R(\bG'))\otimes_{\Omega^*} k^*
\oplus (2,v_1^2)\{y_{12}\}\oplus (4,2v_1, v_1^2)\{y_6y_{12},y_{10}y_{12}\}
\oplus (4,v_1^2)
\{y_{top}\}
\]
where $G'=Spin(11)$. In particular 
$gr_{geo}(X)/2\cong P(1)_2(X)$.
\end{lemma}
Here to prove the above lemma, we used Theorem
14.1 in \cite{YaS}.  However this theorem is also induced from the above lemma, and it seems more easy.
We can see the above lemma as follows.
\begin{proof}[Proof of  Lemma 14.3.]
The image of the restriction map is written (with $mod(I_{\infty}^3)) $  as follows
\[c_6\mapsto 2y_{12},\ \ c_6c_3\mapsto 4y_6y_{12},\ \ 
c_6c_5\mapsto 4y_{10}y_{12},\ \ c_6e_8\mapsto 4y_{top} \]
\[  c_6c_2\mapsto 2v_1y_6y_{12},\ \ 
c_6c_4\mapsto 2v_1y_{10}y_{12},\ \  (e_8c_5\mapsto 4v_1y_{top})\]
\[(c_3c_4-v_1e_8)\mapsto v_1^2y_6y_{12},\ \ c_4c_5\mapsto v_1^2y_{10}y_{12},\ \ 
e_8c_4\mapsto v_1^2y_{top}.\]
Moreover $c_2^2\mapsto v_1^2y_{12}$,
which is zero in $gr_{\gamma}^*(R(\bG))/2$,
since $v_1c_5-2c_4=v_1^2y_{12}$ $mod(I_{\infty}^3)$.

Hence the formula in this lemma is contained in
$Res_{\Omega}(X)$. 
(Note that all generators $c\in A(c)$ in the 
left  hand side in the above appear as $c$ in
the preceding theorem.  The generators $y \in BP^*\otimes P(y)$ in the right hand side appear
in $Res_{\Omega}(X)$ in this lemma.)

For an invariant ideal $J\subset BP^*$, the ideal  $J'=J\otimes_{BP^*} k^*$ is written
\[ (2,v_1^i), \ \ (4,2v_1^i, v_1^j),\quad or \ \  (4,v_1^j).\]
However, we can check 
$ v_1^iy_{12}, 2v_1y_{top} \not \in Res_{\Omega}(X).$
In fact  if $res_{\Omega}(c)=v_1^2y_{12}$, then by dimensional reason,
we see $c=c_2^2=0$, and this is a contradiction.  We
can see $2v_1y_{top}\not \in Res_{\Omega}(X)$ from
$c_5e_8=(2y_{10}+v_1y_{12})(2y_{6}y_{10}+v_1^2y_{6}y_{12})=0\quad mod(I_{\infty}^3).$
\end{proof}

By using $gr_{geo}^*(R(\bG))/2\cong 
P(1)_{\infty}(R(\bG))$ (Lemma 5.4 ), we have the theorem
from computing $\oplus_{y} C(y)$.

We must consider the case $(*)$ in the introduction.
For $I=(i_1,...,i_k)$, let $y_I=y_{i_1}...y_{i_k}$
be the $\bZ_{(p)}$-module generator of $P(y)$.
Here we assume for  $y_I\not =0$ 
\[ (*)\quad |y_I|\not =|y_J|\quad for\ all\ I\not =J.\]

We give the degree for the filtration 
of  $\Omega^*(\bar X)$ 
\[ F^I=F^{|y_I|}=\Omega^*\{y_J ||y_I|\le | y_J|\}.\]
For the invariant ideal $J(y_I)=(a_{I_1},...,a_{I_s})$,we consider the case for some $a_{Ii}$
\[ (**)\quad   a_{I_i}y_I\in I_{\infty}(Im(res_{\Omega}))\cap
\Omega^*\{y_J|J<I\}\quad mod(F^{I+2}).\]

We consider the cases $a_{I_i}=v_1^2$ or $pv_1$ (see also Lemma 4.6).
\begin{lemma}
Suppose $|y_J|=|y_I|+|v_1|$.  Let 
\[ pf_0+v_1f_1=v_1^2y_I \quad (or\ \  pv_1y_I)
\quad in\  \Omega^*(\bar X)/F^{I+2} \]
for $f_0,f_1\in \Omega^*\otimes A(b)$ and 
$f_i \not \in F^I$.  
Moreover $py_J\in Res_{\Omega}(X)$.
Then
\[ f_0=v_1y_J\quad in\ \Omega^*(\bar X)/I_{\infty}^2
\quad (e.g.\ (p,v_1)Res_{\Omega}(X))\]
\end{lemma}
\begin{proof}
Let $pf_0+v_1f_1=v_1^2y_I$.
Hence $f_1$ contains $v_1y_I$ as a summand
i.e., $f_1=\lambda py_J+v_1y_I+...$.  Since
$f_1\not \in F^I$, we can take $\lambda\not =0\mod(p).$
In fact, $b_i\in I_{\infty}\Omega^*(\bar X)$,
so there is $b_i$ such that 
\[f_1=b_i=py_J+v_1y_I\quad mod(I_{\infty}^2\Omega^*(\bar X)).\]

Then we have 
\[ pf_0=v_1^2y_I-v_1f_1=pv_1y_I\quad mod (I_{\infty}^3).\]
and we have the result since $\Omega^*(\bar X)$
is torsion free.
The case $pf_0+v_1f_1=pv_1y_I$ is seen similarly.
(In fact we can take $f_1=0$.)
\end{proof}

Similarly, we can see
\begin{lemma}
Suppose $|y_J|=|y_I|+2|v_1|$ and there is no $K$
such that $|y_K|=|y_I|+|v_1|$.  Let 
\[ pf_0+v_1f_1=v_1^3y_I 
\quad in\  \Omega^*(\bar X)/F^{I+2} \]
for $f_0,f_1\in \Omega^*\otimes A(b)$ and 
$f_i\not \in F^I$.  Then
\[ f_0=v_1y_J\ \ mod(I_{\infty}^2) \quad 
or \quad f_1=p^2y_J+v_1^2y_I \ \ mod(I_{\infty}^3)
\]
\end{lemma}

Now we return the case $G=Spin(13)$.
We seek $\bZ_{(2)}$-module generators $y_J,y_I$ of P(y) such that 
$ |y_J|=|y_I|+|v_1|.$ These cases are  
\[ y_I=y_{12}\ \  (resp. \ y_6y_{12}),\quad y_J=y_{10}\ \ (resp.\ y_6y_{10}).\]

When $y_I=y_{12}$, we have 
seen 
$c_2^2\mapsto v_1^2y_{12}$ and it is zero
in $gr_{geo}^*(R(\bG))/2$.
(Indeed, $ 2c_4-v_1c_5=v_1^2y_{12}$.)
Since $f_0=v_1y_6y_{10}\not \in Res_{\Omega}(X)$,
the assumption of Lemma 14.4 does  not happen
for $y_I=y_6y_{12}$.

\section{The case $G=E_7$ and $p=2$.}

Throughout this section, let
 $(G,p)=(E_7,2)$ the exceptional group of rank $7$.
\begin{thm} (\cite{Mi-Tod}, \cite{Ko-Mi})
The cohomology $grH^*(G;\bZ/2)$ is given 
\[ \bZ/2[y_6,y_{10},y_{18}]
/(y_6^2,y_{10}^2,y_{18}^2),
\otimes \Lambda(z_3,z_5,z_9,z_{15},z_{17},z_{23},z_{27}).\]
\end{thm}
Here the suffix means its degree. We can rewrite
\[ grH^*(G;\bZ/2)\cong \bZ/2[y_1,y_2,y_3]/(y_1^2,y_2^2,y_3^2)\otimes \Lambda(x_1,...,x_7).\]
\begin{lemma} (\cite{Ko-Mi})  The cohomology operations act as
\[\begin{CD}
  x_1=z_3 @>{Sq^2}>> x_2=z_5@>{Sq^4}>>
 x_3=z_9 @>{Sq^8}>> x_4=z_{17}\\
 x_5=z_{15} @>{Sq^8}>> x_6=z_{23}@>{Sq^4}>>
 x_7=z_{27} 
\\
x_5=z_{15}  @>{Sq^2}>> x_4=z_{17} @. @. @.
\end{CD} \]
The Bockstein acts $Sq^1(x_{i+1})=y_i$ for $1\le i\le 3$,
  and
  \[Sq^1\ :\ x_5=z_{15}\mapsto y_1y_2,\ \    x_6=z_{23}\mapsto y_1y_3, \ \  x_{7}=z_{27}\mapsto y_2y_3.\]
\end{lemma}

\begin{lemma} ([Ya5,6])  
In $H^*(G/T)/(4)$, for all monomials $u\in P(y)^+/2$,
except for $y_{top}=y_1y_2y_3$, the elements $2u$ 
are written as elements in
$H^*(BT)$.  Moreover, in $BP^*(G/T)/I_{\infty}^2$,
there are $b_i\in BP^*(BT)$ such that
\[ b_2=2y_1,\ \ b_3=2y_2,\ \ b_4=2y_3,\ \ 
 b_6=2y_1y_3,\ \ b_7=2y_2y_3,\]
\[b_1=v_1y_1+v_2y_2+v_3y_3,\quad b_5=2y_1y_2+v_1y_3.\]
\end{lemma}
\begin{proof}  The last two equations are given by
$ Q_1(x_1)=y_1,\   Q_2(x_1)=y_2,$
\[\  Q_3(x_1)=y_3,\
 Q_1(x_5)=y_3,\  Q_0(x_5)=y_1y_2.\]
\end{proof}

\begin{lemma} (\cite{To1}) We have $t(E_7)_{(2)}=2^2$.
\end{lemma}
\begin{proof}
We get the result from $b_2b_7=(2y_1)(2y_2y_3)=2^2y_{top}$.
\end{proof}

By Chevalley's theorem, $res_{K}$ is surjective
(moreover isomorphic).  Hence in $K^*(\bar R(\bG))/2$,  we see
\[ v_1^sy_2=b\in \bZ/2[b_1,...,b_{\ell}] \quad some \ s\in \bZ.\]
The facts that $|b_i|\ge 6$ and $|y_2|=10$
imply $s=2$.
\begin{lemma}  We can take $b_2=2y_1+v_1^2y_2$
in $ k^*(G/T)/(v_1^3)$.
\end{lemma}

 \begin{prop}  ([YaG])
There is a filtration
whose associated graded ring is 
\[grK^*(R(\bG))  \cong  
K^*\otimes (B_1\oplus B_2)\quad where \]
\[ \begin{cases}
B_1=P(b)/2\oplus 
   \bZ/2\{b_1b_7\},\quad where\ P(b)=\Lambda_{\bZ}(b_1,b_2,b_5)   \\ 
  B_2=\bZ_{(2)}\{2,2b_1,b_3,b_4,b_1b_3,b_6,b_7,b_2b_7\}
\end{cases}\]
\end{prop}

\begin{proof}
We have isomorphisms 
\[ K^*(R(\bG))\cong K^*\otimes P(b)
\ \cong K^*\otimes P(y)\cong K^*(\bar R(\bG)).
\]
Let us write $\bar b_2=b_2-2v_1^{-1}b_1$ and 
$\bar b_5=b_5-v_1^{-1}b_1b_3$ so that
$\bar b_2=v_1^2y_2$, $\bar b_5=v_1y_3$ 
in $K^*(\bar R(\bG))$.  
Moreover let  $P(\bar b)=\Lambda(b_1,\bar b_2,\bar b_5)$.

Let us write  
$ gr_2K^*(R(\bG))=K^*(R(\bG))/2\oplus 2K^*(R(\bG))
,$  and we will compute $gr(gr_2K^*(R(\bG))$ for
some filtration of $gr_2K^*(R(\bG))$.

We still have $K^*(R(\bG))/2\cong K^*\otimes P(b)/2$. The elements in $2K^*(R(\bG))$ are 
computed as
\[K^*(R(\bG))\supset 2K^*\otimes P(b)= 2K^*\otimes
P(\bar b)\]
 \[=2K^*\{1,b_1,\bar b_2,..., \bar b_2\bar b_3,
b_1\bar b_2\bar b_3\}= 2K^*\{1,y_1,y_2,...,y_1y_2y_3\}\]
\[ =K^*\{2,2b_1,b_3,b_4,b_1b_3,b_6,b_7,b_1b_7\}.\]
Here we used 
$2v_1y_1=2b_1,\ \ 2y_2=b_3,\ \  2y_3=b_4,\ \ 2v_1y_1y_2=b_1b_3,\ $ and
\[ 2y_1y_3=b_6,\ \ 2y_2y_3=b_7,\ \ 2v_1y_1y_2y_3=b_1b_7.\]

For the $2$-torsion element $b_1b_7$ in $ grK^*(R(\bG))$, we have 
   $2b_1b_7=4v_1y_1y_2y_3=v_1b_2b_7$ in $K^*(R(\bG))$.
Hence we can write 
\[  grK^*(R(\bG))=gr_2K^*(R(\bG))/(2b_1b_7)\oplus \bZ\{b_2b_7\},\]
which gives the graded ring in this lemma.
\end{proof}

By dimensional reason we have with $mod(I_{\infty}^3)$
\[ (*)\quad b_6=2y_1y_3+\lambda v_1^2y_2y_3,\quad for\ \lambda=0\ or\ 1.\]
However I can not decide this $\lambda$ here.

Suppose that $\lambda=0$.
Then we get  the restriction map $res_{\Omega}$
\[ (b_2,b_1)\to(2,v_1)\{y_1\},
\quad (b_3, 2b_1-v_1b_2)\to (2,v_1^3)\{y_2\},\quad   
 (b_4,b_1b_3-v_1b_5)\to (2,v_1^2)\{y_3\},\]
\[ (b_5,b_1b_2)\to(2,v_1^3)\{y_1y_2\},\quad (b_6,b_1b_5)\to (2, v_1^2)\{y_1y_3\},\quad  
\]\[  (b_7,b_1b_4-b_2b_5)\to (2,v_1^3)\{y_2
y_3\}, \quad (b_2b_7,b_1b_7,b_1b_2b_5)
\to (4,2v_1,v_1^4)\{y_{top}\}.\]
For examples, we can compute $mod(I_{\infty}^4)$
\[ b_1b_2=v_1y_1(2y_1+v_1^2y_2)=2\mu v_1^3y_1y_2+v_1^3y_1y_2\  \ for
\  \mu \in \bZ/2,\]
\[b_1b_4-b_2b_5=2v_1y_1y_3-(2y_1+v_1^2y_2)(2y_1y_2+v_1y_3)=v_1^3y_2y_3.\]         
Here we used $y_1^2=\mu v_1^2y_1y_2+...$ by dimensional reason.
(Remark that $b_1b_2\mapsto v_1^2y_1y_3+v_1v_2y_2y_3$
in $\Omega^*(\bar R(\bG))\ mod(I_{\infty}^3)$.
In particular, $r_{2\Delta_1}(b_1b_2)=0\ mod(I_{\infty}^3)$.)

Suppose $\lambda=1$.  Then we can compute
\[ b_1b_6=v_1y_1(2y_1y_3+v_1^2y_2y_3)=v_1^3y_{top}.\]
Hence we only need to change one  place
\[ (
 b_2b_7,b_1b_7,b_1b_6)\to (4,2v_1,v_1^3)\{y_{top}\}.\]
 
 We can not decide $gr_{geo}(R(\bG))/2$ here,
while  we have

\begin{thm} 
Let $B=\bZ/2\{1,b_2,...,b_7\}.$ 
Then we have  additively        
\[gr_{geo}(R(\bG))/2
\cong \begin{cases}
   B\oplus  B\{b_1\}\oplus \bZ/2\{b_2b_7\}\quad \\
\quad or \ \ 
B\oplus B/(b_6)\{b_1\}
\oplus \bZ/2\{b_2b_7,b_1b_2b_5\}.
\end{cases} \]
\end{thm}

\begin{cor}
The image $Res_{\Omega}(R(\bG))\otimes_{BP^*}k^*$ is isomorphic to the sum
of the right hand side $(y)\subset \Omega^*\otimes P(y)$
in the above restriction map
\[ Res_{\Omega}(R(\bG))\otimes _{BP^*}k^*\cong k^*\{1\}\oplus (2,v_1)\{y_1\}
\]
\[ \oplus (2,v_1^3)\{y_2,y_1y_2,y_2y_3\}\oplus (2,v_1^2)\{y_3,y_1y_3\}\oplus (4,2v_1,v_1^{4-\lambda})\{y_{top}\}. \]
\end{cor}

Now we will check $(*)$ in the introduction (see also Lemma 4.6 and Lemma 14.4, 14.5).  We seek $\bZ_{(2)}$-module generators $y_J,y_I$ of 
$P(y)$ such that 
$v_1^2y_I\in Res_{\Omega}$ and
$|y_J|=|y_I|-2.$  This case is 
$(y_J,y_I)=(y_1y_2,y_3)$. We see 
$v_1y_1y_2\not \in Res_{\Omega}$, and $(*)$ does not happen
from Lemma 14.4.

Next suppose that $v_1^3y_I\in Res_{\Omega}$ and 
$|y_J|+4=y_I$ (see Lemma 14.5). Then
\[ (y_J,y_I)=(y_1,y_2)\ \  and \ \  (y_1y_{3},y_2y_3).\]
In fact, for the first case, we still have 
$2b_1-v_1b_2\mapsto v_1^3y_2$.  For the second case we see that
\[ v_1y_1y_3\not \in Res_{\Omega}\quad and \quad 
2^2y_1y_2+v_1^2y_2y_3 \not \in Res_{\Omega}.\]
Hence $(*)$ does not happen  for this case.

Similarly we can prove that
it does not happen that $pf_0+v_1f_1=v_1^4y_{top}$.
When $y_I=y_{12}$, we have 
seen 
$c_2^2\mapsto v_1^2y_{12}$ and it is zero
in $gr_{geo}^*(R(\bG))/2$.
(Indeed, $ 2c_4-v_1c_5=v_1^2y_{12}$.)
Since $v_1y_6y_{10}\not \in Res_{\Omega}(X)$,
the assumption of Lemma 14.4 does  not happen.

\section{The case $G=E_8$ and $p=2$}

Throughout this section, let $G=E_8$ and $p=2$.
It is known (\cite{Mi-Tod}) that
\[grH^*(G;\bZ/2)
\cong \bZ/2[y_1,y_2,y_3,y_4]/(y_1^8,y_2^4,y_3^2,y_4^2)\otimes \Lambda(x_1,...,x_8)\]
with $|x_8|=29, |y_4|=30 $ so that 
\[ grH^*(E_7,;\bZ/2)\cong grH^*(G,\bZ/2)/(y_1^2,y_2^2,y_4,x_8).\]
 Hence we can write
\[grP(y)\cong grP(y)'\otimes \Lambda(y_1^2,y_1^4,
y_2^2,y_4), \quad grP(y)'=\Lambda(y_1,y_2,y_3).\]

The cohomology operations for $H^*(G;\bZ/2)$ are almost same as that of $H^*(E_7;\bZ/2)$, which are given also by Lemma 15.2 (with $mod(y_1^2,y_2^2,y_4)$).
Moreover, we have
\[Sq^2(x_7)=x_8, \quad Sq^1(x_8)=y_4.\]
Thus we have (see Lemma 15.2) 
\[ (*)\quad  Q_3(x_5)=Q_2(x_6)=Q_1(x_7)=Q_0(x_8)=y_4,\quad |y_4|=30.\]

We can not decide $Res_{\Omega}(R(\bG))$ here,
but only give
an example
\begin{lemma} We have $J(y_1^2)\cong(4,2v_1,v_1^2)$.
\end{lemma} 
\begin{proof} We can see $J(y_1)=(2,v_1)$ as the case 
$E_7$. So $J(y_1^2)\supset (2,v_1)^2=(4,2v_1,v_1^2)$.
But elements $2,v_1$ are not contained in $J(y_1^2)$.
\end{proof}

\begin{lemma}  Let $G'=E_7$ (and $G=E_8$).  Then
       \[ P(1)_1(R(\bG))\cong P(1)_1(R(\bG'))\oplus
(2)\{y_4\}\quad\]
 where $P(1)_1(R(\bG'))\cong
(2,v_1)\{y_1\}\oplus (2)\{y_2,y_3\}\oplus (2)\{y_iy_j|
1\le i<j\le 3\}.$ 
\end{lemma}
\begin{proof} This isomorphism follows from
that
\[ Q_0(x_i)=y \quad or \quad Q_1(x_i)=y'\quad for\ y,y'\in P(y)\]
implies $y\in \Lambda(y_1,y_2,y_3)$ if $i\le 7$, and
$Q_0(x_8)=y_4$, and $y'\in \Lambda(y_1,y_2,y_3)$
for $i\le 6$ and $Q_1(x_7)=y_4$.
\end{proof}

Let $K$ be an extension of $k$ such that
$X$ does not split over $K$ but splits over an extension
over $K$ of degree $2a$, $(2,a)=1$.  Suppose that
\[ (*)\quad y_1,y_2,y_3 \in Res_{\Omega}(R(\bG)|_{K})\quad but \quad
\it y_4\not \in Res_{\Omega}(R(\bG)|_K).\]
(Compare the above condition $(*)$ 
with the condition $(*)$ in $\S 10$.)
That is, the $J$-invariant $J(\bG_K)=(0,0,0,1)$
and such $K$ exists (see [Pe-Se-Za], [Se]).
Then we have the following theorem by arguments similar to those to get Theorem 7.12.
(The motive $R(\bG_k)|_K$ in the theorem 
is an example of motives 
given in Lemma 8.4 in [Se].)
\begin{thm}  Let $X=R(\bG)|_K$.  Then we have 
the isorphism and the surjection  
\[ Res_{\Omega}(X)\cong 
\bZ_{(2)}[y_1,y_2,y_3]/(y_1^8,y_2^4,y_3^2)\otimes (BP^*\{1\}\oplus I_4\{y_4\}),\]  
\[ CH^*(X)/2\twoheadrightarrow 
\bZ/2[y_1,y_2,y_3]/(y_1^8,y_2^4,y_3^2)\otimes CH^*(R_4)/2,\]
where $ CH^*(R_4)/2\cong M_0(1)\oplus M_4(y_4)\cong  \bZ/2\{1,2y_4,v_1y_4,v_2y_4,v_{3}y_4\}.$
The restriction map is given as
$ b_j\mapsto v_{8-j}y_{4}$ if  $5\le j \le 8$, and 
$b_j\mapsto 0$ if  $1\le j\le 4$.
\end{thm}
\begin{proof}
From $(*)$, we have the relations in  $\Omega^*(\bar X)/I_{\infty}^2$
\[ b_5=2y_1y_2+...+v_3y_8,\quad b_6=2y_1y_3+...+v_2y_8,\quad,...\]
Here if $y\in \Lambda(y_1,y_2,y_3),$ then $y\in
Res_{\Omega}(X)$.
Hence $v_3y_8\in Res_{\Omega}(X)$, and so 
in $I_4y_8$.  Thus we have $J(y_8)=I_4$.
\end{proof}

\section{The case $G=E_8$ and $p=3$}

We consider cases $H^*(G)$ has $p$-torsion for an odd prime $p$.
Then $(G,p)$ is of type $(I)$ or $(II)$.  For 
$G$ of type $(I)$, by Petrov-Semenov-Zinoulline,
the motive $R(\bG)\cong R_2$ and are studied in
$\S 4$ detailedly.  For examples
\[Res_{\Omega}(R_2)=BP^*\{1\}\oplus (p,v_1)\{y,...,y^{p-1}\}
\subset \Omega^*[y]/(y^p),\]
\[ CH^*(R_2)/p\cong \bZ/p\{1\}\oplus M_2(y)\{1,...,y^{p-2}\}.\]

In this section, we study the case of type$(II)$,
that is $(G,p)=(E_8,3)$.  Hereafter this section,
we assume $(G,p)=(E_8,3)$.

  The cohomology $H^*(G;\bZ/3)$ is isomorphic to
(\cite{Mi-Tod}) 
\[ \bZ/3[y_{8},y_{20}]/(y_8^3,y_{20}^3)\otimes
\Lambda(z_3,z_7,z_{15},z_{19},z_{27},z_{35},z_{39},z_{47}).\]
By Kono-Mimura \cite{Ko-Mi} the actions of cohomology 
operations
are also known
\begin{thm} (\cite{Ko-Mi})
We have $P^3y_8=y_{20}$, and 
\[ \beta: z_7\mapsto y_8,\ \ z_{15}\mapsto y_8^2,\ \ 
z_{19}\mapsto y_{20},\ \
 z_{27}\mapsto y_{8}y_{20},\ \ z_{35}\mapsto y_{8}^2y_{20},\]
\[
z_{39}\mapsto y_{20}^2,\ \ z_{47}\mapsto y_8y_{20}^2,\] 
\[ P^1:\ z_3\mapsto z_7,\ \  z_{15}\mapsto z_{19},\ \  z_{35}
\mapsto z_{39}\] \[  P^3:\ z_7\mapsto z_{19},\ \  z_{15}\mapsto z_{27}
\mapsto -z_{39},\ \
 z_{35}\mapsto z_{47}.\]
\end{thm}
We use notations $y=y_8,y'=y_{20}$, and $x_1=z_3,...,x_8=z_{47}$.
Then we can rewrite the isomorphisms 
\[ H^*(G;\bZ/3)\cong \bZ/3[y,y']/(y^3,(y')^3)\otimes 
\Lambda(x_1,...,x_{8}), \]
\[ grH^*(G/T;\bZ/3)\cong \bZ/3[y,y']/(y^3,(y')^3)\otimes 
S(t)/( b_{1}, ,...,b_{8}).\]
From Corollary 2.2, we have 
\begin{cor} (\cite{YaG})
We can take $b_1\in BP^*(BT)$ such that 
 \[v_1y+v_2y'=b_1\quad in\  BP^*(G/T)/\II.\]
\end{cor} 

From the preceding theorem, we know that
all $y^i(y')^j$ except for $(i,j)=(0,0)$ and $(2,2)$ are $\beta$-image.
Hence we have 
\begin{cor} (\cite{YaG})
For all nonzero monomials  $u\in P(y)^+/3$
except for $(yy')^2$,  it holds  $3u\in S(t)$.
In fact, in  $BP^*(G/T)/I_{\infty}^2$
\[ b_1=v_1y+v_2y',\ \ b_2=3y,\ \ b_3=3y^2+v_1y' ,\ \ b_4=3y',\]
\[ b_5=3yy',\ \ b_6=3y^2y'+v_1(y')^2,\ \ b_7=3(y')^2,\ \ b_8=3y(y')^2.\]
 \end{cor}

In the paper \cite{YaG},  we compute
\[gr_{\gamma}^*(G/T)/3\cong gr_{geo}^*(\bG/B_k)/3.\]
We first compute some graded ring of $K^*(R(\bG))$. 
\begin{prop}  (Proposition 9.7 in \cite{YaG})
  There is a filtration whose 
associated graded ring is 
\[gr'K^*(R(\bG))  \cong 
K^*\otimes (B_1\oplus B_2)\quad where \]
\[ \begin{cases}
B_1=(P(b)/(3)\oplus 
   \bZ/3\{b_1b_2,b_1b_8\}),\quad P(b)=\bZ_{(3)}[b_1,
 b_3]/(b_1^3, b_3^3),   \\ 
  B_2=\bZ_{(3)}\{3,b_2,b_4,b_5,b_6,b_7,b_8\}
\oplus \bZ_{(3)}\{b_2b_8\}.
\end{cases}\]
\end{prop}

{\bf Remark.}
 The arguments in the last lines in the page 151 
in \cite{YaG} were
not correct. 
The element $\{b_2^2\}$ is zero in $gr_{\gamma}(R(\bG))/2$.  Delete $b_2^2$ in 
the module in Lemma 9.6 and $B$ in Proposition 9.7.

\begin{thm}  (Theorem 9.12 in \cite{YaG})
For $B_1,B_2$ in Proposition 9.7, we have
\[ gr_{geo}(R(\bG))/3\cong
B_1/(3,b_1b_3^2,b_1^2b_3^2)\oplus
\bZ/3\{b_1b_6,b_1^2b_6\}\oplus B_2/3.\]
\end{thm}

\begin{cor}
We have 
\[ gr_{geo}(R(\bG))/3\cong B_1'\oplus B_1''\oplus
B_2\]
\[ where\quad  \begin{cases}
B_1'
\cong \bZ/3[b_1,b_3]/(b_1^3,b_2^3,b_1b_3^2,b_1^2b_3^2)
 \cong \bZ/3\{1,b_1,b_3,b_1^2, b_1b_3,b_3^2,b_1^2b_3\},\\
B_1''\cong \bZ/3\{ b_1b_2,b_1b_8,b_1b_6,b_1^2b_6\}, \\
B_2\cong \bZ/3\{ b_2,b_4,b_5,b_6,b_7,b_8,b_2b_8\}.
\end{cases} \]
\end{cor}

The restriction maps are given as 
\[ (b_2,b_1) \to (3,v_1)\{y\},\quad 
   (b_3,b_1^2) \to (3,v_1^2)\{y^2\},\quad 
(b_4,b_1b_2-v_1b_3)\to (3,v_1^2)\{y'\},\]
\[ (b_5, b_1b_3) \to (3,v_1^2)\{yy'\},\quad 
   (b_6, b_1^2b_3) \to (3,v_1^3)\{y^2y'\},\quad 
(b_7, b_3^2-2v_1b_6)\to (3,v_1^2)\{(y')^2\},\]
\[ (b_8,b_1b_6) \to (3,v_1^2)\{y(y')^2\},\qquad
(b_2b_8,b_1b_8,b_1^2b_6)\to (9,3v_1,v_1^3)\{(yy')^2\}.\]
Note all elements $b_i...b_j$ in $B_1',B_1'',B_2$
in the above corollary
appear in the left hand side of the above maps.
\begin{lemma} We have the isomorphism
\[ Res_{\Omega\ }(R(\bG))\otimes _{P^*}k^*\cong k^*\{1\}\oplus (3,v_1)\{y\}
 \oplus (3,v_1^2)\{y^2,y',yy',(y')^2,y(y')^2\} \]
\[ \oplus (3,v_1^3)\{y^2y'\}\oplus (9,3v_1,v_1^3)\{y_{top}\}.\]
Hence $gr_{\gamma}(R(\bG)/3 \cong 
P(1)_3(R(\bG))$. \end{lemma}
Thus we can also prove the above theorem (and corollary) from the above lemma and the arguments for
$C(y)$.

For ease of arguments, we write $d(x)=1/4|x|$, e.g.,
\[ d(v_1)=-1,\ d(b_1)=1,\ d(b_2)=2,\ d(b_3)=4,\ d(b_4)=5 \]
\[ d(b_5)=7,\ d(b_6)=9,\ d(b_7)=10,\ d(b_8)=12.\]

\begin{proof}[Proof of Lemma 17.7.]
Let $X=R(\bG)$.
From the above restriction maps,  we see 
the right hand side ideals are contained in $Res_{\Omega}(X)$.
We can see $J(y^2)=(3,v_1^2)$.  Otherwise
$J(y^2)=(3,v_1)$ and $v_1y^2\in Res_{\Omega}(X)$. Then 
there is a ring generator $x_i\in H^*(G;\bZ/3)$ such that $Q_0(x_i)=0,Q_1(x)=y^2$ This is a contradiction.
Similarly we can see $J=(2,v_1^2)$ for cases
where $v_1^2$ appears.
 
Therefore 
we only need to see
\[  v_1^2y^2y',\ \ v_1^2y_{top} \ \  are\ not\ in \ Res_{\Omega}(X).\]

Suppose $x=v_1^2y^2y'\in Res_{\Omega}(X)$.
Since $x$ $mod(I_{\infty}^3)$ is a sum of products of $b_i$ and $b_j$
(since each $b_i\in I_{\infty}$). 
Since $ d(x)=-2+4+5=7$, we see 
$x=b_2b_4=9yy'$ $mod(I_{\infty}^3).$
 This is a contradiction.
Suppose $x=v_1^2y_{top}\in Res_{ \Omega}(X).$
Then $d(x)=-2+4+10=12$.  Hence
$x=b_4b_5$ or $b_2b_7$.  We see
$ b_4b_5=9y(y')^2=b_2b_7$. But $9y(y')^2=3b_8$ and
it is not a $BP^*$-module generator in $Res_{\Omega}(X)$.
\end{proof}

Now we will see $(*)$ in the introduction for  the case $G=E_8$, $p=3$.
(Recall Lemma 14.4, 14,5.)
We seek $\bZ_{(2)}$-module generators $y_J,y_I$ of $P(y)$ such that 
$v_1^2y_I\in Res_{\Omega}$ and
$|y_J|=|y_I|-4$ Then .
\[(y_J,y_I)=(y^2,y')\quad and \quad (y^2y',(y')^2).\]
Since $v_1y^2,v_1y^2y'\not \in Res_{\Omega}$, this cases do not happen.

Next suppose that $v_1^3y_I\in Res_{\Omega}$ and 
$|y_J|=|y_I|-8$. Then
\[ (y_J,y_I)=(yy',y^2y' )\ \  and \ \  (y(y')^2, (yy')^2.\]
We can see neither case does not happen
from Lemma14.5.
For example, $v_1yy'\not \in Res_{\Omega}$ and 
$3^2yy'+v_1^2y^2y'\not \in Res_{\Omega}$.
This shows the first case.

\end{document}